\documentclass[10pt]{amsart}

\usepackage{amssymb}
\usepackage{mathrsfs}
\usepackage{amscd}
\usepackage{verbatim}
\usepackage{hyperref}

\theoremstyle{plain}
\newtheorem{theorem}{Theorem}[section]
\theoremstyle{remark}
\newtheorem{remark}[theorem]{Remark}

\newtheorem{example}[theorem]{Example}
\theoremstyle{plain}
\newtheorem{corollary}[theorem]{Corollary}
\newtheorem{lemma}[theorem]{Lemma}
\newtheorem{proposition}[theorem]{Proposition}

\numberwithin{equation}{section}

\def\R{{\mathbb R}}
\def\C{{\mathbb C}}

\newcommand{\E}{{\mathbb E}}
\renewcommand{\P}{{\mathbb P}}
\newcommand{\F}{{\mathcal F}}

\renewcommand{\H}{{\mathcal H}}

\renewcommand{\a}{\alpha}

\newcommand{\g}{\gamma}

\newcommand{\e}{\varepsilon}
\renewcommand{\l}{\lambda}
\renewcommand{\o}{\omega}
\renewcommand{\O}{\Omega}

\newcommand{\LH}{L^2(0,T;H)}

\newcommand{\ggH}{\g(\LH,E)}

\newcommand{\tr}{{\rm Tr}}

\newcommand{\beq}{\begin{equation}}
\newcommand{\eeq}{\end{equation}}
\newcommand{\bal}{\begin{aligned}}
\newcommand{\eal}{\end{aligned}}
\newcommand{\ben}{\begin{enumerate}}
\newcommand{\een}{\end{enumerate}}
\newcommand{\bit}{\begin{itemize}}
\newcommand{\eit}{\end{itemize}}

\newcommand{\bth}{\begin{theorem}}

\newcommand{\bpr}{\begin{proposition}}
\newcommand{\epr}{\end{proposition}}
\newcommand{\ble}{\begin{lemma}}
\newcommand{\ele}{\end{lemma}}
\newcommand{\bpf}{\begin{proof}}
\newcommand{\epf}{\end{proof}}
\newcommand{\bex}{\begin{example}}
\newcommand{\eex}{\end{example}}
\newcommand{\bre}{\begin{example}}
\newcommand{\ere}{\end{example}}

\newcommand{\D}{{\mathcal D}}
\newcommand{\calL}{{\mathcal L}}
\newcommand{\n}{\Vert}
\newcommand{\one}{{{\bf 1}}}

\newcommand{\s}{^*}
\newcommand{\lb}{\langle}
\newcommand{\rb}{\rangle}

\newcommand{\limn}{\lim_{n\to\infty}}

\newcommand{\sumn}{\sum_{n\ge 1}}

\begin{document}

\title[Regularity for the Zakai equation]
{It\^o's formula in UMD Banach spaces and regularity of
solutions of the Zakai equation}

\author{Z. Brze\'zniak}
\address{Department of Mathematics\\University of York\\
York, YO10 5DD\\ England}
\email{zb500@york.ac.uk}

\author{J.M.A.M. van Neerven}
\address{Delft Institute of Applied Mathematics\\
Technical University of Delft \\ P.O. Box 5031\\ 2600 GA Delft\\The Netherlands}
\email{J.M.A.M.vanNeerven@tudelft.nl}

\author{M.C. Veraar}
\address{Delft Institute of Applied Mathematics\\
Technical University of Delft \\ P.O. Box 5031\\ 2600 GA Delft\\The Netherlands}
\email{M.C.Veraar@tudelft.nl}

\author{L. Weis}
\address{Mathematisches\, Institut\, I \\
\, Technische\, Universit\"at \, Karlsruhe \\
\, D-76128 \, Karls\-ruhe\\Germany} \email{Lutz.Weis@mathematik.uni-karlsruhe.de}

\thanks{The second and third named authors are supported by
a `VIDI subsidie' (639.032.201) in the `Vernieuwingsimpuls' programme
of the Netherlands Organization for Scientific Research (NWO).
The second named author is also supported by
a Research Training Network (HPRN-CT-2002-00281).
The fourth named author is supported by grants from the Volkswagenstiftung
(I/78593) and the Deutsche Forschungsgemeinschaft (We 2847/1-1).}

\keywords{Stochastic integration in Banach spaces, UMD spaces, It\^o formula,
stochastic evolution equations, Zakai equation,
non-autonomous equations, Wong-Zakai approximation}

\subjclass[2000]{Primary: 60H15 Secondary: 28C20, 35R60, 46B09, 60B11}

\date\today

\begin{abstract}
Using the theory of stochastic integration for
processes with values in a UMD Banach space developed recently by the authors,
an It\^o formula is proved which is
applied to prove the existence of strong solutions for a class of
stochastic evolution equations in UMD Banach spaces.
The abstract results are applied to prove regularity in space and time
of the solutions of the Zakai equation.
\end{abstract}

\maketitle

\section{Introduction}
In this paper we study space-time regularity of strong solutions of the
nonautonomous Zakai equation

\begin{equation}\label{eq:zakai}
\begin{aligned}
D_t U(t,x) & = A(t,x, D) U(t,x) + B(x,D) U(t,x) D_t W(t), \ \ t\in [0,T], x\in
\R^d
\\ U(0,x) & = u_0(x), \ \ x\in \R^d.
\end{aligned}
\end{equation}
Here
\[
\bal
A(t,x, D)
& = \sum_{i,j=1}^d a_{ij}(t,x) D_i D_j  + \sum_{i=1}^d q_{i}(t,x) D_i  + r(t,x), \\ B(x,D)
& =  \sum_{i=1}^d b_{i}(x) D_i + c(x).
\eal
\]
This equation arises in filtering theory,
and has been studied by many authors, cf. \cite{ATito,DPZ,Z}
and the references therein. It can be written as an abstract stochastic
evolution equation of the form
\begin{equation}\label{eq:abstr-zakai}\begin{aligned}
d U(t) &= A(t) U(t) dt + B U(t)\, d W(t), \qquad t\in [0,T],
\\ U(0) & = u_0.
\end{aligned}
\end{equation} Here the linear operators $A(t)$ are closed and densely
defined on a suitable Banach space $E$, the operator $B$ is a generator of a
$C_0$-group on $E$, and $W$ is a real-valued Brownian motion on some probability space
$(\O,\F,\P)$.

In the framework where $E$ is the Hilbert space $L^2(\R^d)$, the autonomous
version of the problem \eqref{eq:abstr-zakai} has been studied for instance by
Da Prato, Iannelli and Tubaro \cite{DPIT}
and Da Prato and Zabczyk \cite{DPZ}, who proved the existence of strong
solutions for this equation. By applying the results to the Zakai equation
\eqref{eq:zakai} and assuming that $u_0\in L^2(\R^d)$ almost surely, under
suitable regularity conditions on the coefficients the existence of solutions
with paths in
$$C([0,T];L^2(\R^d))\cap C((0,T];W^{2,2}(\R^d))$$
is established. If $u_0\in
W^{2,2}(\R^d)$ almost surely, then the solution has paths in
$C([0,T];W^{2,2}(\R^d))$.

In the slightly different setting of a Gelfand triple of separable Hilbert spaces, a class of problems including \eqref{eq:abstr-zakai} was studied with the same method by
Brze\'zniak, Capi\'nski and Flandoli \cite{BCF}. For Zakai's equation
they obtain solutions in the space $C([0,T];L^2(\R^d))\cap L^2(0,T;W^{1,2}(\R^d))$
for initial values $u_0\in L^2(\R^d)$.

Using different techniques,  Brze\'zniak \cite{Brz1} studied a class of equations containing the
autonomous case $A(t) \equiv A$ of \eqref{eq:abstr-zakai}
in the setting of martingale type $2$ spaces $E$. For $E = L^p(\R^d)$ with $2\le p<\infty$ and initial values
$u_0$ taking values almost surely in the Besov space $B_{p,2}^1(\R^d)$, the existence of solutions for the autonomous Zakai equation with paths
in $L^2(0,T;W^{2,p}(\R^d))$ and continuous moments in $B_{p,2}^1(\R^d)$ was obtained.
The techniques of
\cite{DPIT} cannot be extended to the setting of martingale type $2$
spaces $E$, since this would require an extension of the It\^o
formula for the duality mapping. Here the problem arises that if $E$
has martingale type $2$, then $E\s$ has martingale type $2$ only if
$E$ is isomorphic to a Hilbert space (see \cite{Kwa,Pi2}).

The method of \cite{DPIT} reduces the stochastic
problem \eqref{eq:abstr-zakai} to a certain deterministic problem. Crucial to
this approach is the use of It\^o's formula for bilinear forms on Hilbert
spaces. This method has been extended by Acquistapace and Terreni \cite{ATito}
to the nonautonomous case using the Kato-Tanabe theory \cite[Section 5.3]{Ta1}
for operators $A(t)$ with time-dependent domains. In this approach, a
technical difficulty arises due to the fact that in the associated
deterministic problem, certain operator valued functions are only H\"older
continuous, whereas the Kato-Tanabe theory requires their differentiability.
This difficulty is overcome by approximation arguments. The authors also note
that for the case where the domains $\D(A(t))$ do not depend on time, the
methods from \cite{DPIT} can be extended using the Tanabe theory \cite[Section
5.2]{Ta1}.

In the present paper we will extend the techniques of \cite{DPIT} to UMD
spaces $E$. This class of spaces includes $L^p(\R^d)$ for $p\in (1,\infty)$.
The extension relies on the fact that if $E$ is a UMD space, then $E\s$ is a
UMD spaces as well. Using the theory of stochastic integration in UMD spaces
developed recently in \cite{NVW1}, an It\^o formula is proved which is
subsequently applied to the duality mapping defined on the UMD space $E\times
E\s$, $(x,x\s)\mapsto \lb x,x\s\rb$. For the Zakai equation with initial value
$u_0\in L^p(\R^d)$ almost surely, where $1<p<\infty$, this results in
solutions with paths belonging to
$$C([0,T];L^{p}(\R^d))\cap C((0,T];W^{2,p}(\R^d)).$$
If
$u_0\in W^{2,p}(\R^d)$ almost surely,
the solution has paths in $C([0,T];W^{2,p}(\R^d)).$
For initial values in $L^p(\R^d)\cap L^\infty(\R^d)$ (respectively, in
$W^{2,p}(\R^d)\cap W^{2,\infty}(\R^d)$) for some $1\le p<\infty$,
the Sobolev embedding theorem then gives solutions with paths in
$C((0,T];C^{1,\a}(\R^d))$ (respectively,
 in $C([0,T];C^{1,\a}(\R^d))$)
for all $\a\in (0,1)$. If $u_0$ takes its values in a certain
interpolation space between $L^p(\R^d)$ and $W^{2,p}(\R^d)$, we
obtain that the solution has paths in
$$
C([0,T];L^{p}(\R^d))\cap L^q(0,T;W^{2,p}(\R^d)),$$
for appropriate $q\in [1, \infty)$.

Rather than using the Kato-Tanabe theory for operators $A(t)$ with
time-depen\-dent domains, we shall use the more recent Acquistapace-Terreni
theory developed in \cite{AT2}. The above-mentioned technical difficulty does
not occur then.

Another approach was taken by Krylov \cite{Kry}, who developed an $L^p$-theory
for a very general class of time-dependent parabolic stochastic partial
differential equations on $\R^d$ by analytic methods. For Zakai's equation
with initial conditions $u_0$ in the Bessel potential space $H^{r+2-\frac2p,p}(\R^d)$,
where $r\in \R$ and $2\le p<\infty$,
solutions are obtained with paths in $$L^p(0,T;H^{r+2,p}(\R^d)).$$

Further $L^p$-regularity results for the Zakai equation may be found in
\cite{Kot,KR,Par2}.

\section{It\^o's formula in UMD Banach spaces}

We start with a brief discussion of the $L^p$-theory of stochastic integration
in UMD Banach spaces developed recently in \cite{NVW1}.
We fix a separable real Hilbert space $H$ and a real Banach space $E$, and
denote by
$\calL(H,E)$ the space of all bounded linear operators from $H$ to $E$.

Let $(\O,\F,\P)$ be a probability space and let $F$ be a Banach space.
An $F$-valued {\em random variable} is a strongly
measurable mapping on $\O$ into $F$. The vector
space of all $F$-valued random variables on $\O$, identifying random
variables if they agree almost surely, is denoted by $L^0(\O;F)$. We
endow $L^0(\O;F)$ with the topology induced by convergence in
probability.

An $F$-valued {\em process} is a one-parameter family of random
variables with values in $F$. Often we identify a
process with the induced mapping $I\times \O\to F$, where $I$ is the
index set of the process. In most cases below, $I= [0,T]$.
A process $\Phi: [0,T]\times\O\to \calL(H,E)$ will be called {\em $H$-strongly
measurable} if for all $h\in H$ the process $\Phi h: [0,T]\times\O\to E$
defined by $\Phi h(t,\o) := \Phi(t,\o)h$, is strongly measurable.

For a separable real Hilbert space $\H$, let
$\g(\H,E)$ denote the operator ideal of $\g$-radonifying operators in
$\calL(\H,E)$. Recall that $R\in \calL(\H,E)$ is {\em $\g$-radonifying} if for
some (equivalently, for each) orthonormal basis $(h_n)_{n\ge 1}$ the Gaussian
sum $\sumn \g_n Rh_n$ converges in $L^2(\O;E)$. Here, $(\g_n)_{n\ge 1}$ is a
sequence of independent real-valued standard Gaussian random variables on $\O$.
We refer to \cite{DJT,NVW1, NW1,Nh} for its definition and
relevant properties. Below we shall be interested primarily in the
case $\H=L^2(0,T;H)$.

An $H$-strongly measurable process
$\Phi:[0,T]\times\O\to\calL(H,E)$ is said to {\em represent} a random variable
$X\in L^0(\O;\ggH)$ if for all $x^*\in E^*$, for almost all $\o\in\O$ we have
$\Phi\s(\cdot,\o)x\s \in\LH$ and
\begin{equation}\label{eq:repr}
\lb X(\o)f, x\s\rb = \int_0^T [f(t),
\Phi\s(t,\o)x\s]_H\,dt \ \ \hbox{for all $f\in\LH$}.
\end{equation}
Strong measurability of $X$ can usually be checked with \cite[Remark
2.8]{NVW1}. If $\Phi$ represents both $X_1, X_2\in L^0(\O;\ggH)$,
then $X_1=X_2$ almost surely by the Hahn-Banach theorem and the
essential separability of the ranges of $X_1$ and $X_2$. In the converse
direction, if both $\Phi_1$ and $\Phi_2$ represent $X\in
L^0(\O;\ggH)$, then $\Phi_1 h=\Phi_2 h$ almost everywhere on
$\o\times[0,T]$ for all $h\in H$ (to see this take $f =
1_{[a,b]}\otimes h$ in \eqref{eq:repr}; then use the Hahn-Banach
theorem and the strong $H$-measurability of $\Phi$) and therefore
$\Phi_1 = \Phi_2$ almost everywhere on $\o\times[0,T]$. It will
often be convenient to identify $\Phi$ with $X$ and we will simply
write $\Phi\in L^0(\O;\ggH)$.

From now on we shall assume that a filtration $(\F_t)_{t\in [0,T]}$ on $(\O,\F,\P)$ is given which satisfies the usual conditions.
A process $\Phi:[0,T]\times\O\to \calL(H,E)$ is called an {\em
elementary process adapted to $(\F_t)_{t\in [0,T]}$} if it can be written as
$$
\Phi(t,\omega) = \sum_{n=0}^N \sum_{m=1}^M \one_{(t_{n-1},t_n]\times
A_{mn}}(t,\omega) \sum_{k=1}^K h_k \otimes x_{kmn},
$$
where $0 \le t_0 < \dots< t_N \le T$ and the sets
$A_{1n},\dots,A_{Mn}\in \F_{t_{n-1}}$ are disjoint for each $n$
(with the understanding that  $(t_{-1},t_0] := \{0\}$ and
$\F_{t_{-1}} := \F_0$) and the vectors $h_1,\dots,h_K \in H$ are
orthonormal. For such $\Phi$ we define the stochastic integral
process with respect to $W_H$ as an element of $L^0(\O;C([0,T];E))$
as
\[
t\mapsto \int_0^{t} \Phi(t) \, d W_H(t) =
\sum_{n=0}^N \sum_{m=1}^M \one_{A_{mn}}(\omega) \sum_{k=1}^K (W_H(t_n\wedge t)h_k - W_H(t_{n-1}\wedge t)h_k) x_{kmn}
\]
Here $W_H$ is a
cylindrical Brownian motion. For a process
$\Phi:[0,T]\times\Omega\to \calL(H,E)$ we say that $\Phi$ is {\em
scalarly in $L^0(\O;\LH)$} if for all $x^*\in E^*$, for almost all
$\o\in\O$ we have $\Phi\s(\cdot,\o)x\s \in\LH$. The following result
from \cite{NVW1} extends the integral to a larger class of
processes.

\begin{proposition}\label{prop:main1loc} Assume that
$E$ is a UMD space and let $W_H$ be an $H$-cylindrical Brownian motion on
$(\O,\F,\P)$. For an $H$-strongly measurable and adapted process
$\Phi:[0,T]\times\Omega\to \calL(H,E)$ which is scalarly in $L^0(\O;\LH)$ the
following assertions are equivalent:
\begin{enumerate}
\item\label{sequence} there exist elementary adapted processes
$\Phi_n: [0,T]\times\O\to\calL(H,E)$ such that:
\begin{enumerate}
\item[(i)] for all $h\in H$ and $x\s\in E\s$,
$$\lb \Phi h,x\s\rb = \limn
\lb \Phi_n h, x\s \rb \ \ \hbox{in measure};$$
\item[(ii)]
there exists a process $\zeta\in L^0(\O;C([0,T];E))$ such that
$$ \zeta = \limn \int_0^\cdot \Phi_n(t)\,dW_H(t) \quad\hbox{in $L^0(\O;C([0,T];E))$}.$$
\end{enumerate}
\item\label{weak} There exists a process $\zeta\in L^0(\O;C([0,T];E))$ such that for all
$x\s\in E\s$,
$$ \lb \zeta,x\s\rb = \int_0^\cdot \Phi\s(t)x\s\,dW_H(t) \ \ \hbox{in
$L^0(\O;C[0,T])$}.$$
\item\label{gamma} $\Phi\in L^0(\O;\ggH)$.
\end{enumerate}
The processes $\zeta$ in \eqref{sequence} and \eqref{weak} are
indistinguishable and it is uniquely determined as an element of
$L^0(\O;C([0,T];E))$. It is a continuous local martingale starting
at $0$, and for all $p\in(1, \infty)$ there exists a constant
$0<C_{p,E}<\infty$ such that
\[C_{p,E}^{-1} \E \n \Phi\n_{\ggH}^p \leq \E\sup_{t\in [0,T]}\n\zeta(t)\n^p
\leq C_{p,E} \E \n \Phi\n_{\ggH}^p.\]
\end{proposition}

A process $\Phi:[0,T]\times\O\to\calL(H,E)$ satisfying the equivalent
conditions of the theorem  will be called {\em stochastically integrable} with
respect to $W_H$. The process $\zeta$ is called the {\em stochastic integral
process} of $\Phi$ with respect to $W_H$, notation
$$ \zeta = \int_0^\cdot \Phi(t)\,dW_H(t).$$

The following lemma will be needed in Section \ref{appl} and
shows that condition \eqref{weak} in Proposition \ref{prop:main1loc} can be
weakened.

\begin{lemma}\label{Lpw*denseloc}
Let $E$ be a UMD Banach space and let $F$ be a dense subspace of $E^*$. Let
$\Phi:[0,T]\times\O\to \calL(H, E)$ be an $H$-strongly measurable and adapted
process such that for all $x^*\in F$, $\Phi^* x^*\in L^2(0,T;H)$ almost
surely. If there exists process $\zeta\in L^0(\O;C([0,T];E))$ such that for
all $x^*\in F$ we have
\begin{equation}\label{Fextendloc} \lb \zeta, x^*\rb =
\int_0^\cdot \Phi^*(s) x^* \, d W_H(s) \ \ \text{in} \ L^0(\O;C[0,T]),
\end{equation} then $\Phi$ is stochastically integrable with respect to $W_H$ and
\[\zeta =\int_0^\cdot \Phi(s) \, d W_H(s) \  \ \text{in} \ L^0(\O;C([0,T];E)).\]
\end{lemma}
\begin{proof}
By Proposition \ref{prop:main1loc}, it suffices to show that $\Phi^* x^*\in
L^2(0,T;H)$ almost surely and that\eqref{Fextendloc} holds for all $x^*\in
E^*$. To do so, fix $x^*\in E^*$ arbitrary and choose elements $x_n^*\in F$
such that $x^* = \limn x_n^*$ in $E\s$. Clearly we have $\lb \zeta, x^*\rb =
\limn \lb \zeta, x^*_n\rb$ in $L^0(\O;C[0,T])$. An application of
\cite[Proposition 17.6]{Kal} shows that the processes $\Phi^* x_n^*$ define a
Cauchy sequence in $L^0(\O;L^2(0,T;H))$. By a standard argument we obtain that
$\Phi\s x\s\in L^0(\O;L^2(0,T;H))$ and $\limn \Phi^*x_n^*=\Phi^*x^* $ in
$L^0(\O;L^2(0,T;H))$. By another application of \cite[Proposition 17.6]{Kal}
we conclude that
\[
\int_0^\cdot \Phi^*(s) x^* \, d W_H(s) = \limn \int_0^\cdot
\Phi^*(s) x^*_n \, d W_H(s) = \limn \lb \zeta, x_n^*\rb = \lb \zeta, x^*\rb
\]
in $L^0(\O;C[0,T])$.
\end{proof}

The next lemma defines a trace which will be needed in the statement of the
It\^o formula.

\begin{lemma}\label{hulptr}
Let $E,F,G$ be Banach spaces and let $(h_n)_{n\geq 1}$ be an orthonormal basis
of $H$. Let $R\in \g(H, E)$, $S\in \g(H, F)$ and $T\in \calL(E,\calL(F,G))$ be
given. Then the sum
\begin{equation}\label{trdef}
\tr_{R,S} T:=\sum_{n\geq 1} (T R h_n)(S h_n)
\end{equation}
converges in $G$ and does not depend on the choice of the orthonormal basis. Moreover,
\begin{equation}\label{trineq}
\|\tr_{R,S} T\|\leq \|T\| \|R\|_{\g(H,E)} \|S\|_{\g(H,F)}.
\end{equation}
\end{lemma}
If $E=F$ we shall write $\tr_R := \tr_{R,R}$.

\begin{proof}
First assume that $S = \sum_{n=1}^{N} h_n \otimes y_n$ with $y_1,
\ldots, y_N\in F$. Then the convergence of the series in
\eqref{trdef} is obvious. Letting $\xi_R = \sum_{n=1}^{N} \g_n R h_n
$ and $\xi_S = \sum_{n=1}^{N} \g_n S h_n$ we obtain
$$
\|\tr_{R,S} T\| = \|\E T(\xi_R) (\xi_S)\|
\leq \|T\|(\E\|\xi_R\|^2)^{\frac12} (\E\|\xi_S\|^2)^{\frac12} \leq
\|T\|\|R\|_{\g(H,E)} \|S\|_{\g(H,F)}.
$$
Now let $S\in \g(H,F)$ be arbitrary. For each $N\geq 1$, let $P_N\in
\calL(H)$ denote the orthogonal projection on span$\{h_n: n\leq
N\}$. Letting $S_n = S \circ P_n$, we have $S = \limn S_n$ in
$\g(H,F)$. For all $m,n\geq 1$, we have
\[\|\tr_{R,S_n} T - \tr_{R,S_m} T\| = \|\tr_{R,S_n-S_m} T\|\leq \|T\| \|R\|_{\g(H,E)} \|S_n-S_m\|_{\g(H,F)}.\]
Therefore, $(\tr_{R,S_n} T)_{n\geq 1}$ is a Cauchy sequence in $G$,
and it converges. Clearly, for all $N\geq 1$, $\tr_{R,S_N} T =
\sum_{n=1}^N (T R h_n)(S h_n)$. Now the convergence of \eqref{trdef}
and the estimate \eqref{trineq} follow.

Next we show that the trace is independent of the choice of the
orthonormal basis $(h_n)_{n\geq 1}$. Let $(e_n)_{n\geq 1}$ be another orthonormal
basis for $H$. For $R = \sum_{m=1}^{M} h_m \otimes x_m$ with $x_1,
\ldots, x_M\in E$ and $S = \sum_{n=1}^{N} h_n \otimes y_n$ with $y_1,
\ldots, y_N\in F$, we have
\[\begin{aligned}
\sum_{k\geq 1}  T(R e_k) (S e_k) & = \sum_{k\geq 1} \sum_{m\geq 1}
\sum_{n\geq 1} [e_k,h_m] [e_k,h_n] T(R h_m) (S h_n)
\\ & = \sum_{m=1}^M \sum_{n=1}^N \sum_{k\geq 1} [e_k,h_m] [e_k,h_n] T(R h_m) (S h_n)
\\ & =\sum_{m=1}^M \sum_{n=1}^N \delta_{mn} T(R h_m) (S h_n) = \tr_{R,S} T.
\end{aligned}\]
The general case follows from an approximation argument as before.
\end{proof}

A function $f:[0,T]\times E\to F$ is said to be of {\em class
$C^{1,2}$} if $f$ is differentiable in the first variable and twice
Fr\'echet differentiable in the second variable and the functions
$f$, $D_1 f$, $D_2 f$ and $D_2^2 f$ are continuous on $[0,T]\times
E$. Here $D_1f$ and $D_2f$ are the derivatives with respect to the
first and second variable, respectively. We proceed with a version
of It\^o's formula as announced in \cite{NVW1}.

\begin{theorem}[It\^o formula]\label{thm:itoformula}
Let $E$ and $F$ be UMD spaces. Assume that $f:[0,T]\times E\to F$ is
of class $C^{1,2}$. Let $\Phi: [0,T]\times\O\to\calL(H,E)$ be an
$H$-strongly measurable and adapted process which is stochastically
integrable with respect to $W_H$ and assume that the paths of $\Phi$
belong to $L^2(0,T;\g(H,E))$ almost surely. Let $\psi:
[0,T]\times\O\to E$ be strongly measurable and adapted with paths in
$L^1(0,T;E)$ almost surely. Let $\xi:\O\to E$ be strongly
$\mathcal{F}_0$-measurable. Define $\zeta:[0,T]\times\O\to E$ by
$$
\zeta=\xi + \int_0^\cdot \psi(s) \, ds+ \int_0^\cdot \Phi(s) \, dW_H(s).
$$
Then $s\mapsto D_2 f(s, \zeta(s)) \Phi(s)$ is stochastically integrable and
almost surely we have, for all $t\in[0,T]$,
\begin{equation}
\begin{aligned} f(t, \zeta(t))-f(0, \zeta(0)) = & \int_0^t D_1 f(s, \zeta(s)) \, ds + \int_0^t D_2
f(s, \zeta(s)) \psi(s)\, ds
\\ & \qquad + \int_0^t D_2 f(s, \zeta(s)) \Phi(s)\, dW_H(s)
\\ & \qquad + \tfrac{1}{2}\int_0^t \tr_{\Phi(s)} \big(D_2^2 f(s, \zeta(s)) \big) \,ds.
\end{aligned}\label{eq:itoformula}
\end{equation}
\end{theorem}
The first two integrals and the last integral are almost surely defined as a
Bochner integral. To see this for the last integral, notice that by Lemma
\ref{hulptr} we have
\[\begin{aligned}
\int_0^t \big\|\tr_{\Phi(s)} \big(D_2^2 f(s, \zeta(s)) \big) \big\| \, ds
&\leq \int_0^t \|D_2^2 f(s, \zeta(s))\| \|\Phi(s)\|^2_{\g(H,E)} \, ds\\ & \leq
\sup_{s\in [0,T]}\|D_2^2 f(s, \zeta(s))\| \|\Phi\|_{L^2(0,T;\g(H,E))}^2
\end{aligned}\]
almost surely.

\begin{remark}
In the situation of Theorem \ref{thm:itoformula}, Via Proposition
\ref{prop:main1loc}, the stochastic integrability implies that $\Phi\in
L^0(\O;\ggH)$. If, in addition to the assumptions of Theorem
\ref{thm:itoformula}, we assume that $E$ has type $2$, then
\[L^2(0,T;\g(H,E))\hookrightarrow \ggH\]
canonically. Therefore, the assumption that $\Phi$ is stochastically
integrable is automatically fulfilled since $\Phi\in
L^0(\O;L^2(0,T;\g(H,E)))$. In that case the theorem reduces to the It\^o
formula in \cite{Brz3,Nh}.

If $E$ has cotype $2$, then
$$\ggH\hookrightarrow L^2(0,T;\g(H,E))$$
canonically and the assumption that $\Phi\in L^0(\O;L^2(0,T;\g(H,E)))$ is
automatically fulfilled if $\Phi$ is stochastically integrable.
\end{remark}

As a consequence of Theorem \ref{thm:itoformula} we obtain the following
corollary.
\begin{corollary}\label{ex:dual}
Let $E_1$, $E_2$ and $F$ be UMD Banach spaces and let $f:E_1\times E_2\to F$ be
a bilinear map. Let $(h_n)_{n\geq 1}$ be an orthonormal basis of $H$.
For $i=1,2$ let $\Phi_i: [0,T]\times\O\to\calL(H,E_i)$,
$\psi_i: [0,T]\times\O\to E$ and $\xi_i:\O\to E_i$
satisfy the assumptions of Theorem \ref{thm:itoformula} and define
$$
\zeta_i(t)=\xi_i + \int_0^t \psi_i(s) \, ds+ \int_0^t \Phi_i(s) \,
dW_H(s).
$$
Then, almost surely for all $t\in [0,T]$,
\[\begin{aligned}
f(\zeta_1(t), \zeta_2(t)) - f(\zeta_1(0), \zeta_2(0)) = & \int_0^t
f(\zeta_1(s), \psi_2(s)) + f(\psi_1(s), \zeta_2(s)) \, ds
\\ & \quad + \int_0^t f(\zeta_1(s), \Phi_2(s)) + f(\Phi_1(s), \zeta_2(s))\, dW_H(s)
\\ & \quad + \int_0^t \sum_{n\geq 1} f(\Phi_1(s) h_n, \Phi_2(s) h_n) \,ds.
\end{aligned}\]
\end{corollary}

In particular, for a UMD space $E$, taking $E_1 = E$, $E_2=E^*$, $F=\R$ and
$f(x, x^*) = \lb x, x^*\rb$, it follows that almost surely for all
$t\in[0,T]$,
\begin{equation}\label{itodual}
\begin{aligned}
\lb \zeta_1(t), \zeta_2(t)\rb - \lb \zeta_1(0), \zeta_2(0)\rb = & \int_0^t \lb
\zeta_1(s), \psi_2(s)\rb + \lb \psi_1(s), \zeta_2(s)\rb \, ds
\\ & \qquad + \int_0^t \lb \zeta_1(s), \Phi_2(s)\rb + \lb \Phi_1(s), \zeta_2(s)\rb\, dW_H(s)
\\ & \qquad + \int_0^t \sum_{n\geq 1} \lb \Phi_1(s) h_n, \Phi_2(s) h_n\rb \,ds.
\end{aligned}
\end{equation}

The result of Corollary \ref{ex:dual} for martingale type $2$ spaces $E_1$,
$E_2$ and $F$ can be found in \cite[Corollary 2.1]{Brz3}. However, we want to
emphasize that it is not possible to obtain \eqref{itodual} with martingale
type $2$ methods, since $E$ and $E^*$ have martingale type $2$ if and only if
$E$ is isomorphic to a Hilbert space.

For the proof of theorem \ref{thm:itoformula} we need two lemmas.

\begin{lemma}\label{lem:phin}
Let $E$ be a UMD space. Let $\Phi:[0,T]\times\O\to\calL(H,E)$ be an
$H$-strongly measurable and adapted process which is stochastically integrable
with respect to $W_H$ and assume that its paths belong to $L^2(0,T;\g(H,E))$
almost surely. Then there exists a sequence of elementary adapted processes
$(\Phi_n)_{n\geq 1}$ such that
\[\Phi=\limn \Phi_n \ \text{in} \ L^0(\O;L^2(0,T;\g(H,E)))\cap
L^0(\O;\ggH).\]
\end{lemma}
\begin{proof}
Let $(h_n)_{n\geq 1}$ be an orthonormal basis for $H$ and denote by $P_n$ the
projection onto the span of $\{h_1, \ldots, h_n\}$ in $H$. Define
$\Psi_n:[0,T]\times \O\to \g(H,E)$ as
$$\begin{aligned}
\Psi_n(t, \omega) h :&=\E(R^{\delta_n}(\Phi(\cdot, \omega) P_n
h)|\mathcal{G}_n)(t) \\ &= \sum_{k=1}^{2^n}\one_{((k-1)2^{-n}T,
k2^{-n}T]}(t)\int_{(k-2) 2^{-n}T}^{(k-1)2^{-n} T} \Phi(s) P_n h \,
ds ,
\end{aligned}
$$ where $R^{\delta_n}:L^2(0,T;E)\to L^2(0,T;E)$ denotes right
translation over $\delta_n=2^{-n}$, $\mathcal{G}_n$ is the $n$-th
dyadic $\sigma$-algebra.
By \cite[Proposition
2.1]{NVW1}, $\Phi = \limn \Psi_n$ in $L^0(\O;L^2(0,T;\g(H,E)))$ and
$\Phi = \limn {\Psi_n}$ in $L^0(\O;\gamma(L^2(0,T;H), E))$.

The processes $\Psi_n$ are not elementary in general, but of the form
$$\Psi_n=\sum_{k=1}^{K_n} 1_{(t_{kn}, t_{k+1,n}]}\otimes
\sum_{i=1}^n h_i \otimes\xi_{ikn},$$ where
each $\xi_{ikn}$ is an $\mathcal{F}_{t_{kn}}$-measurable $E$-valued
random variable. Approximating each $\xi_{ikn}$ in probability by a
sequence of $\mathcal{F}_{t_{kn}}$-simple random variables we
obtain a sequence of elementary adapted processes
$(\Psi_{nm})_{m\geq 1}$ such that $\lim_{m\to\infty}\Psi_{nm}=
\Psi_n$ in $L^0(\O;L^2(0,T;\g(H,E)))$ and
$\lim_{m\to\infty}\Psi_{nm}\!=\!{\Psi_n}$ in
$L^0(\O;\gamma(L^2(0,T;H), E))$. For an appropriate subsequence
$(m_n)_{n\geq 1}$, the elementary adapted processes $\Phi_{n m_n}$
have the required properties.
\end{proof}

The next lemma is proved in a similar way.
\begin{lemma}\label{lem:psin}
Let $E$ be a Banach space, and let $\psi\in L^0(\O;L^1(0,T;E))$ be an adapted
process. Then there exists a sequence of elementary adapted processes
$(\psi_n)_{n\geq 1}$ such that $\psi = \limn \psi_n$ in $L^0(\O;L^1(0,T;E))$.
\end{lemma}

\begin{proof}[Proof of Theorem \ref{thm:itoformula}]
The proof is divided into several steps.

{\em Step 1} -- Reduction to the case $F=\R$. Assume the theorem holds in the
case $F=\R$. Applying this to $\lb f, x^*\rb$ for $x^*\in E^*$ arbitrary we
obtain
\[\begin{aligned} \lb f(t, \zeta(t)), x^*\rb-\lb f(0, \zeta(0)), x^*\rb
&=  \Big< \int_0^t D_1 f(s, \zeta(s)) \, ds, x^*\Big>
\\ & \qquad + \Big< \int_0^t D_2
f(s, \zeta(s)) \psi(s)\, ds, x^*\Big>
\\ & \qquad + \int_0^t \big(D_2 f(s, \zeta(s)) \Phi(s)\big)^{*} x^*\, dW_H(s)
\\ & \qquad + \frac{1}{2}\Big<\int_0^t \tr \big(D_2^2 f(s, \zeta(s))
(\Phi(s),\Phi(s))\big)
\,ds,x^*\Big>.
\end{aligned}\]
An application of Proposition \ref{prop:main1loc} (\ref{weak}) to the
pathwise continuous process
\[\begin{aligned}
\ & f(\cdot, \zeta)-  \lb f(0, \zeta(0)) - \int_0^\cdot D_1 f(s,
\zeta(s)) \, ds - \int_0^\cdot D_2 f(s, \zeta(s)) \psi(s)\, ds \\ &
\qquad \qquad - \frac{1}{2}\int_0^\cdot \tr \big(D_2^2 f(s, \zeta(s))
(\Phi(s),\Phi(s))\big) \,ds.
\end{aligned}\]
shows that $D_2 f(\cdot, \zeta) \Phi$ is stochastically integrable
and \eqref{eq:itoformula} holds. It follows that it suffices to
consider $F=\R$.

{\em Step 2} -- Reduction to elementary adapted processes. Assume the theorem
holds for elementary processes. By path continuity it suffices to show that
for all $t\in [0,T]$ almost surely \eqref{eq:itoformula} holds. Define the
sequence $(\zeta_n)_{n\geq 1}$ in $L^0(\O;C([0,T];E))$ by
$$
\zeta_n(t)=\xi_n + \int_0^t \psi_n(s) \, ds+ \int_0^t \Phi_n(s) \, dW_H(s),
$$
where $(\xi_n)_{n\geq 1}$ is a sequence of
$\mathcal{F}_0$-measurable simple functions with $\xi = \limn \xi_n$
almost surely and $(\Phi_n)_{n\geq 1}$ and $(\psi_n)_{n\geq 1}$ are
chosen from Lemma \ref{lem:phin} and \ref{lem:psin}. By
\cite[Theorems 5.5 and 5.9]{NVW1} we have $\zeta=\limn \zeta_n$ in
$L^0(\O;C([0,T];E))$. We may choose $\O_0\subseteq \O$ of full
measure and a subsequence which we again denote by $(\zeta_n)_{n\geq
1}$ such that
\begin{equation}\label{zetan}
\zeta=\limn \zeta_n \ \text{in} \ C([0,T];E) \ \text{ on $
\O_0$}.
\end{equation}
Thus, in order to prove \eqref{eq:itoformula} holds for the triple $(\xi, \psi, \Phi)$ it
suffices to show that all terms in \eqref{eq:itoformula} depend continuously
on $(\xi, \psi, \Phi)$. This is standard, but we include the details for
convenience.

For the left hand side of \eqref{eq:itoformula} it follows from \eqref{zetan}
that
$$\limn f(t, \zeta_n(t))-f(0, \zeta_n(0)) =  f(t, \zeta(t))-f(0, \zeta(0)) \ \ \text{almost
surely.}$$

For a continuous function $p:[0,T]\times E\to B$, where $B$ is some Banach
space, and $\omega\in \Omega_0$ fixed the set
\[\{p(s, \zeta_n(s,\omega)): \ s\in[0,T], \ n\geq 1\}\cup \{p(s, \zeta(s,\omega)):\ s\in[0,T]\}\]
is compact in $B$, hence bounded. Let $K = K(\omega)$ denote the maximum of
these bounds obtained by applying this to the functions $f$, $D_2f$ and
$D_2^2f$. By Lemma \ref{lem:psin}, \eqref{zetan} and dominated convergence, on
$\Omega_0$ we obtain
\[\limn\int_0^t D_1 f(s, \zeta_n(s)) \, ds
  = \int_0^t D_1 f(s, \zeta(s)) \, ds,\]
\[\limn\int_0^t D_2 f(s, \zeta_n(s)) \psi_n(s) \, ds
= \int_0^t D_2 f(s, \zeta(s)) \psi(s) \,ds.\]

For the stochastic integral term in \eqref{eq:itoformula}, by \cite[Lemma
17.12]{Kal} it is enough to show that on $\Omega_0$,
\begin{equation}\label{eq:Kal}
\limn \|D_2 f(\cdot, \zeta) \Phi -  D_2 f(\cdot,\zeta_n) \Phi_n
\|_{L^2(0,T;H)}= 0.
\end{equation}
Here $D_2 f(\cdot, \zeta)$ and $D_2 f(\cdot, \zeta_n)$ stand for $D_2 f(\cdot,
\zeta(\cdot))$ and $D_2 f(\cdot, \zeta_n(\cdot))$, respectively. But, by Lemma
\ref{lem:phin} we have
\[\begin{aligned}\lim_{n\to\infty} \|D_2 f(\cdot, \zeta_n)(\Phi-\Phi_n) \|_{L^2(0,T;H)} & \le
K \limn \|\Phi-\Phi_n\|_{L^2(0,T;\calL(H,E))}\\ & \leq K \limn
\|\Phi-\Phi_n\|_{L^2(0,T;\g(H,E))}= 0,
\end{aligned}\] and, by \eqref{zetan} and dominated
convergence,
$$\limn \|(D_2 f(\cdot, \zeta)- D_2 f(\cdot, \zeta_n)) \Phi\|_{L^2(0,T;H)}= 0$$
on $\Omega_0$. Together these estimates give \eqref{eq:Kal}.

For the last term in \eqref{eq:itoformula} we have
\[
\begin{aligned}
\ &  \|\tr_{\Phi} (D_2^2 f(\cdot, \zeta)) - \tr_{\Phi_n} (D_2^2 f(\cdot,
\zeta_n) )\|_{L^1(0,T)}
\\ & \qquad\qquad \leq \|\tr_{\Phi}(D_2^2 f(\cdot, \zeta))- \tr_{\Phi}(D_2^2 f(\cdot, \zeta_n))\|_{L^1(0,T)}
\\ & \qquad\qquad\qquad + \|\tr_{\Phi}(D_2^2 f(\cdot, \zeta_n)) - \tr_{\Phi_n} (D_2^2 f(\cdot,
\zeta_n))\|_{L^1(0,T)}.
\end{aligned}
\]
The first term tends to $0$ on $\Omega_0$ by Lemma \ref{hulptr}, \eqref{zetan}
and dominated convergence. For the second term, by Lemma \ref{hulptr}, the
Cauchy-Schwartz inequality and Lemma \ref{lem:phin} we have
\[
\begin{aligned}
\ & \|\tr_{\Phi} (D_2^2 f(\cdot, \zeta_n)) - \tr_{\Phi_n}(D_2^2 f(\cdot,
\zeta_n)) \|_{L^1(0,T)}
 \\ & \qquad \qquad \leq  \|\tr_{\Phi} (D_2^2 f(\cdot, \zeta_n))
- \tr_{\Phi, \Phi_n}(D_2^2 f(\cdot, \zeta_n))\|_{L^1(0,T)}
 \\ & \qquad \qquad \qquad + \|\tr_{\Phi, \Phi_n}((D_2^2
f(\cdot, \zeta_n)) - \tr_{\Phi_n}(D_2^2 f(\cdot, \zeta_n))\|_{L^1(0,T)}
\\ & \qquad \qquad \leq
K\|\Phi\|_{L^2(0,T;\g(H,E))}\|\Phi-\Phi_n\|_{L^2(0,T;\g(H,E))}
\\ & \qquad\qquad\qquad +K\|\Phi_n\|_{L^2(0,T;\g(H,E))}\|\Phi-\Phi_n\|_{L^2(0,T;\g(H,E))},
\end{aligned}
\]
which tends to $0$ on $\Omega_0$ as well.

{\em Step 3 --} If $\xi$ is simple, $\psi$ and $\Phi$ are
elementary, they take their values in a finite dimensional subspace
$E_0\subseteq E$ and there exists a finite dimensional subspace
$H_0$ of $H$ such that $H = H_0 \oplus \text{Ker}(\Phi)$. Since
$E_0$ is isomorphic to some $\R^n$ and $H_0$ is isomorphic to some
$\R^m$, \eqref{eq:itoformula} follows from the corresponding real
valued It\^o formula.
\end{proof}

\begin{remark}
With more elaborate methods one may show that in Corollary \ref{ex:dual} the
assumption $\Phi\in L^0(\O;L^2(0,T;\g(H,E_i)))$ is not needed. The proof of
this result depends heavily on the fact that $D^2 f$ is constant in that case.
For general functions $f$ of class $C^{1,2}$ we do not know if the assumption
can be omitted.
\end{remark}

\section{The abstract stochastic equation\label{appl}}

After these preparations we start our study of the  problem
\begin{equation}\label{LinSCP}
\begin{aligned}
d U(t) & = A(t) U(t) dt + \sum_{n=1}^N B_n U(t) d W_n(t), \ \ t\in[0,T],\\ U(0) & = u_0.
\end{aligned}
\end{equation}
The processes $W_n= (W_n(t))_{t\in [0,T]}$ are independent standard
Brownian motions on some probability space $(\O,\F,\P)$ and are adapted to some
filtration ${\mathbb F} = (\F_t)_{t\in [0,T]}$. The initial random variable
$u_0:\O\to E$ is assumed to be strongly $\F_0$-measurable.
Concerning the operators $A(t):\D(A(t))\subseteq E \to E$ and
$B_n:\D(B_n)\subseteq E \to E$ we assume the
following hypotheses.

\medskip
\begin{enumerate}
\item[(H1)]\label{H10} The operators
$A(t)$ are closed and densely defined;
\item[(H2)]\label{H11}
The operators $B_n$ generate commuting $C_0$-groups
$G_n = (G_n(t))_{t\in \R}$ on $E$;
\item[(H3)]\label{H12}
For all $t\in [0,T]$ we have $\D(A(t))\subseteq \bigcap_{n=1}^N \D(B_n^2)$.
\end{enumerate}
Defining $\D(C(t)):=\D(A(t))$ and
$C(t):= A(t)- \frac12\sum_{n=1}^N B_n^2,$
we further assume
\begin{enumerate}
\item[(H4)] There exists a $\lambda\in \R$ with $\lambda\in \varrho(A(t))\cap \varrho(C(t))$ for all $t\in [0,T]$,
such that the functions $t\mapsto B_n^2 R(\lambda,A(t))$ and $t\mapsto B_n^2
R(\lambda,C(t))$ are strongly continuous on $[0,T]$.
\end{enumerate}
Hypothesis (H4) is automatically fulfilled in the case $A(t)$ is independent
of $t$. Below it is showed that it is fulfilled in several time dependent
situation as well.

An $E$-valued process $U=\{U(t)\}_{t\in [0,T]}$ is called a {\em strong
solution of \eqref{LinSCP} on the interval $(0,T]$} if $U\in C([0,T];E)$
almost surely, $U(0)=u_0$, and for all $\e>0$ the following conditions are
satisfied:

\begin{enumerate}
\item For almost all $\o\in\O$,
$U(t,\o)\in \D(A(t))$  for almost all $t\in [\e,T]$ and the path $t\mapsto
A(t) U(t,\o)$ belongs to $L^1(\e,T;E)$;
\item For $n=1,\dots,N$ the process $B_n U$ is
stochastically integrable with respect to $W_n$ on $[\e,T]$;
\item Almost surely,
$$
U(t) = U(\e) + \int^t_\e  A(s) U(s) \,ds + \sum_{n=1}^N \int_\e^t B_n U(s)
\,dW_n(s) \ \ \text{for all $t\in [\e,T]$.}
$$
\end{enumerate}
Note that by path continuity, the exceptional sets may be chosen independently
of $\e\in (0,T]$. We call $U$ a {\em strong solution on the interval $[0,T]$}
if $U$ satisfies (1), (2) and (3) with $\e=0$.

Assuming Hypotheses (H1)--(H4), in the Hilbert space setting the existence of
strong solutions has been established in \cite{DPIT} (see also \cite[Section
6.5]{DPZ}) by reducing the stochastic problem to a deterministic one and then
solving the latter by parabolic methods. Here we shall extend this method to
the setting of UMD spaces using the bilinear It\^o formula of the previous
section.

Define $G:\R^N\to \calL(E)$ as $$G(a) := \prod_{n=1}^N
G_n(a_n).$$ Note that each $G(a)$ is invertible with inverse $G^{-1}(a)
:=(G(a))^{-1} = G(-a)$.
For $t\in [0,T]$ and $\o\in\O$ we define the operators $C_W(t,\o):\D(C_W(t,\o))
\subseteq E\to E$ by
$$\begin{aligned}
\D(C_W(t,\o)) & := \{x\in E:\ G(W(t,\o)) x\in \D(C(t))\}, \\
   C_W(t,\o) & := G^{-1}(W(t,\o))C(t)  G(W(t,\o)),
\end{aligned}
$$
where $W=(W_1, \ldots, W_N)$.
Note that the processes $$G_W(t,\o) := G(W(t,\o)) \ \ \hbox{ and } \ \  G_W^{-1}(t,\o) := G(-W(t,\o))$$
are adapted and pathwise strongly continuous.

In terms of the random
operators $C_W(t)$ we introduce the following pathwise problem:
\begin{equation}\label{detproblem}
\begin{aligned}
V'(t) & = C_W(t) V(t), \ \ t\in[0,T],
\\ V(0) &= u_0,
\end{aligned}
\end{equation}
Notice that \eqref{detproblem} is a special case of \eqref{LinSCP} with $A(t)$
replaced by $C_W(t)$ and with $B_n = 0$. In particular the notion of strong
solution on $(0,T]$ and on $[0,T]$ apply.

Note that if $V$ is a strong solution of \eqref{detproblem} on $(0,T]$, then
almost surely we have $G_W(t)V(t)\in\D(C(t)) =
\D(A(t))\subseteq\bigcap_{n=1}^N \D(B_n^2)$ for almost all $t\in [0,T]$.

\medskip
The next theorem, which extends \cite[Theorem 1]{DPIT} and
\cite[Theorem 1]{DPT} to UMD Banach spaces, relates the problems \eqref{LinSCP}
and \eqref{detproblem}.

\begin{theorem}\label{solequiv}
Let $E$ be a UMD Banach space and assume {\rm (H1)--(H4)} and let $\e\in
[0,T]$ be fixed. For a strongly measurable and adapted process
$V:[0,T]\times\O\to E$ the following assertions are equivalent:
\begin{enumerate}
\item $G_W V$ is a strong solution of \eqref{LinSCP} on $(0,T]$  (resp. on
$[0,T]$);
\item $V$ is a strong solution of \eqref{detproblem} on $(0,T]$ (resp. on
$[0,T]$).
\end{enumerate}
\end{theorem}

\begin{proof}
First we claim that $\bigcap_{m,n=1}^N \D(B_n^{*} B_m^*)$ is
norm-dense in $E^*$. Since $E$ is reflexive it is sufficient to
prove that $\bigcap_{m,n=1}^N \D(B_n^{*} B_m^*)$ is weak$\s$-dense
in $E\s$. Fix an $x\in E$, $x\not=0$, and some $\lambda\in
\bigcap_{n=1}^N\varrho(B_n)$, and put $y := \prod_{n=1}^N R(\lambda,
B_n)^2 x$. Since by (H2) the resolvents $R(\l,B_n)$ commute we have
$y\in  \D\Big(\prod_{n=1}^N B_n^2\Big)$. Since $y\neq 0$ we can find
$y^*\in E^*$ such that $\lb y, y^*\rb \neq 0$. Then by (H2), the
resolvents $R(\lambda,B_n^*)$ commute and $x^* := \prod_{n=1}^N
R(\lambda,B_n^*)^2 y^*\in \bigcap_{m,n=1}^N \D(B_n^{*}B_m^{*})$ and
it is obvious that $\lb x, x^*\rb \neq 0$. This proves the claim.

We will now turn to the proof of the equivalence of strong solutions on
$(0,T]$. The equivalence of strong solutions
on $[0,T]$ follows by taking
$\e=0$ in the proofs below.

(1) $\Rightarrow$ (2): Let $\e>0$ be arbitrary. Since $U:= G_W V$ is a strong
solution of \eqref{LinSCP} on $(0,T]$, almost surely we have $G_W(t)V(t)\in
\D(C(t))$ for almost all $t\in [\e,T]$. Moreover, for $n=1, \ldots, N$,
\[
\begin{aligned}
B^2_n U(t) & = B_n^2 R(\lambda, A(t)) (\lambda - A(t)) U
\\ & =  B_n^2 R(\lambda, A(t)) \lambda U(t) + B_n^2 R(\lambda, A(t)) A(t) U(t).
\end{aligned}
\]
Therefore, (H4) implies that  $B_n^2 G_W V = B_n^2 U$ is in $L^1(\e, T;E)$
almost surely. We conclude that $t\mapsto C(t)G_W(t)V(t)$ belongs to
$L^1(\e,T;E)$ almost surely. Hence $t\mapsto C_W(t) V$ belongs to
$L^1(\e,T;E)$ almost surely.

Let $x^*\in \bigcap_{m,n=1}^N \D(B_n^*B_m^*)$ be fixed. The function
$f:\R^N\to E^*$ defined by $f(a) := G^{-1*}(a)x^*$ is twice
continuously differentiable with
$$\frac{\partial f}{\partial a_n}(a)
= -G^{-1*}(a)B^*_n  x^*, \qquad \frac{\partial^2 f}{\partial
a_n^2}(a) = G^{-1*}(a)B_n^{*2}  x^*.$$  By the It\^o formula Theorem
\ref{thm:itoformula} (applied to the Banach space $E\s$ and the
Hilbert space $H=\R^N$) it follows that the processes $ G_W^{-1*}
B_n^* x^*$ are stochastically integrable with respect to $W_n$ on
$[\e,T]$ and that almost surely, for all $t\in [\e,T]$,
$$
\begin{aligned}
\ & G_W^{-1*}(t) x^* - G_W^{-1*}(\e) x^*
\\ & \qquad
=
-\sum_{n=1}^N \int_\e^t G_W^{-1*}(s) B_n^* x^*\, d W_n(s) + \frac12
\sum_{n=1}^N\int_\e^t G_W^{-1*}(s)  B_n^{2*}x^* \, ds.
\end{aligned}
$$
By \eqref{itodual} applied to $U$ and $G_W^{-1*} x^*$ we obtain that  almost surely, for all
$t\in [\e,T]$,
\[\begin{aligned}
\ &  \lb  V(t), x^*\rb - \lb V(\e), x^*\rb
 = \lb U(t), G_W^{-1*}(t)  x^* \rb -  \lb U(\e), G_W^{-1*}(\e)  x^* \rb
\\ & \qquad = \int_\e^t \frac12  \sum_{n=1}^N \lb U(s), G_W^{-1*}(s) B_n^{*2} x^*\rb +
\lb A(s) U(s), G_W^{-1*}(s) x^*\rb \, ds
\\ & \qquad \qquad + \sum_{n=1}^N \int_\e^t - \lb U(s),   G_W^{-1*}(s)B_n^* x^*\rb
+  \lb B_n U(s), G_W^{-1*}(s) x^*\rb  \, dW_n(s)
\\ & \qquad\qquad - \sum_{n=1}^N \int_\e^t \lb B_n U(s), G_W^{-1*}(s) B_n^*
x^*\rb \, ds
\\ & \qquad  = \int_\e^t \lb G_W^{-1}(s) C(s) U(s), x^*\rb \, ds
\\ &  \qquad
= \int_\e^t \lb C_W(s) V(s), x^*\rb \, ds.
\end{aligned}
\]
Since $C_W V$ has paths in $L^1(\e,T;E)$ almost surely, it follows that,
almost surely, for all $t\in[\e,T]$,
$$
\lb  V(t), x^*\rb - \lb V(\e), x^*\rb  = \Big< \int_\e^t C_W(s) V(s) \, ds,
x^*\Big>
$$
By approximation this identity extends to arbitrary $x^*\in E^*$. By
strong measurability, this shows that almost surely, for all $t\in [\e,T]$,
$$
V(t) - V(\e) = \int_\e^t C_W(s) V(s) \, ds.
$$

(2) $\Rightarrow$ (1):  Put $U:=G_W V$. Let $\e>0$ be arbitrary. Since $V$ is
a strong solution of \eqref{detproblem} on $(0,T]$, as before (H4) implies
that almost surely we have $U(t)\in \D(A(t))$ for all $t\in [0,T]$ and
$t\mapsto A(t)U(t)$ belongs to $L^1(\e,T;E)$.

Let $x^*\in \bigcap_{m,n=1}^N \D(B_n^{*}B_m^*)$ be fixed. Applying
the It\^o formula in the same way as before, the processes
$G_W^*B_n^* x^*$ are stochastically integrable with respect to $W_n$
on $[\e,T]$ and almost surely we have, for all $t\in [\e,T]$,
$$
\begin{aligned} G_W^*(t)x^* - G_W^*(\e) x^*
=\sum_{n=1}^N \int_\e^t  G_W^*(s) B^*_n x^* \, d W_n(s) +  \frac12
\sum_{n=1}^N\int_\e^t G_W^*(s) B_n^{*2} x^* \, d s.
\end{aligned}
$$
By assumption we have $C_W V \in L^1(\e,T;E)$ almost surely. Hence
we may apply \eqref{itodual} with $V$ and $G_W^* x^*$. It follows
that almost surely, for all $t\in [\e, T]$,
\[\begin{aligned}
\ & \lb U(t), x^*\rb - \lb U(\e), x^*\rb
\\ & \qquad = \lb V(t), G_W^*(t) x^*\rb - \lb V(\e), G_W^*(t) x^*\rb
\\ & \qquad = \int_\e^t \frac12 \sum_{n=1}^N \lb V(s),  G_W^*(s) B_n^{*2}x^*\rb +
\lb G_W^{-1}(s) C(s) G_W(s) V(s), G_W^*(s) x^*\rb \, ds
\\ &\qquad  \qquad \qquad  + \sum_{n=1}^N \int_\e^t \lb V(s),  G_W^*(s)B_n^* x^*\rb  \, d W_n(s)
\\ &\qquad  = \int_\e^t \lb A(s) G_W(s)V(s), x^*\rb \, ds
+ \sum_{n=1}^N \int_\e^t \lb B_n G_W(s) V(s), x^*\rb \, d W_n(s)
\\ & \qquad = \int_\e^t \lb A(s) U(s), x^*\rb \, ds
+ \sum_{n=1}^N \int_\e^t \lb B_n U(s), x^*\rb \, dW_n(s).
\end{aligned}
\]
Since $G_W^{-1}CU=C_W V \in L^1(\e,T;E)$ almost surely, we have
$CU\in L^1(\e,T;E)$ almost surely, and therefore by (H4) we also have
$AU\in L^1(\e,T;E)$ almost surely. Also, $V$ has continuous paths almost surely,
and therefore the same is true for $U = G_W V$.
Thanks to the claim we are now in a position to apply Lemma \ref{Lpw*denseloc}
on the interval $[\e,T]$
(for the Hilbert space $H=\R^N$ and the process
$\zeta = U - U(\e) - \int_\e^\cdot A(s)U(s)\,ds$).
We obtain that the processes
$B_n U$ are stochastically
integrable with respect to $W_n$ on $[\e,T]$ and that
almost surely we have, for all $t\in [\e,T]$,
$$
U(t) - U(\e) = \int_\e^t A(s) U(s)\, ds + \sum_{n=1}^N \int_\e^t B_n U(s) \, d
W_n(s).
$$
\end{proof}

\section{The deterministic problem: Acquistapace-Terreni
conditions}\label{sec:AT}

\newcommand{\A}{{\mathcal C}}

Consider the non-autonomous Cauchy problem:
\begin{equation}\label{nCP}
\begin{aligned}
\frac{d u}{dt}(t) &= \A(t) u(t) \ \ t\in [0,T],
\\ u(0) &= x,
\end{aligned}
\end{equation}
where $\A(t): \D(\A(t))\subseteq E\to E$ are linear operators.
We study this equation assuming the
Acquistapace-Terreni conditions \cite{AT2}:

\begin{enumerate}
\item[(AT1)] For all $t\in [0,T]$, $\A(t):\D(\A(t))\subseteq E\to E$
is a closed linear operator and there exists $\theta\in \big( \frac{\pi}{2},
\pi \big)$ such that for all $t\in [0,T]$ we have $$\varrho(\A(t))\supseteq
\overline{S_{\theta}},$$ where $S_{\theta}=\{z\in \C\setminus\{0\}: \ |\arg z|
< \theta\}$. Moreover there exists a constant  $K\geq 0$ such that for all
$t\in [0,T]$
 we have
\[ \n R(\lambda, \A(t)) \n \leq \frac{K}{1+|\lambda|}, \qquad \l\in  S_{\theta}.\]
\item[(AT2)] There exist $k\ge 1$ and constants $L\geq 0$,
$\alpha_1, \ldots, \alpha_k$, and $\beta_1, \ldots, \beta_k\in \R$ with $0\leq
\beta_i<\alpha_i\leq 2$ such that for all $t,s\in [0,T]$ we have
\[\|\A(t) R(\lambda, \A(t)) [\A^{-1}(t) - \A^{-1}(s)]\|\leq L \sum_{i=1}^k |t-s|^{\alpha_i} |\lambda|^{\beta_i-1},\qquad \l\in  S_{\theta}.\]
\end{enumerate}
We may assume $\delta:=\min\{\alpha_i - \beta_i\}\in (0,1)$.

\medskip

We say that $u$ is a
{\em classical solution} of \eqref{nCP} if
\begin{enumerate}
\item $u\in C([0,T];E)\cap C^1((0,T],E)$;
\item $u(t)\in \D(\A(t))$ for all $t\in (0,T]$;
\item  $u(0) = x$ and $u'(t) = \A(t) u(t)$ for all $t\in (0,T]$.
\end{enumerate}

Assuming that $x\in \D(\A(0))$ we say that $u$ is a {\em strict solution} of
\eqref{nCP} if
\begin{enumerate}
\item $u\in C^1([0,T];E)$;
\item $u(t)\in \D(\A(t))$ for all $t\in [0,T]$;
\item  $u(0) = x$ and $u'(t) = \A(t) u(t)$ for all $t\in [0,T]$.
\end{enumerate}

As a special case of \cite[Theorems 6.1, 6.3 and 6.5]{AT2} and \cite[Theorem
5.2]{Acq} we have the following result. For a closed densely defined operator
$(\mathscr{A},\D(\mathscr{A}))$ on $E$ we use the usual notation $D_\mathscr{A}(\theta,p) =
(E,\D(\mathscr{A}))_{\theta,p}$ for the real interpolation spaces.

\begin{theorem}\label{thm:AT}
If the operators $\A(t)-\mu$ satisfy {\rm (AT1)} and {\rm (AT2)} for
some $\mu\in \R$, then the following assertions hold:
\begin{enumerate}
\item If $x\in \overline{\D(\A(0))}$, then there exists a unique classical solution $u$ of
\eqref{nCP}.

\item If $x\in \D_{\A(0)}(1-\sigma, \infty)$ with $\sigma\in (0,1)$, then there exists a unique classical solution $u$ of
\eqref{nCP}. Moreover $\A u\in L^p(0,T;E)$ for all $1\leq
p<\sigma^{-1}$.

\item If $x\in \D(\A(0))$, then there exists a unique strict solution $u$ of
\eqref{nCP}.
\end{enumerate}
\end{theorem}

Assuming Hypothesis (H2), we study the problem
\begin{equation}\label{GhAGh}
\begin{aligned}
\frac{d u}{dt}(t) &= \A_h(t) u(t) \ \ t\in [0,T],
\\ u(0) &= x,
\end{aligned}
\end{equation}
Here $\A_h(t) = G^{-1}(h(t)) \A(t) G(h(t))$, with $\D(\A_h(t)) =
\{x\in E: G(h(t)) x\in \D(\A(t))\}$, $G$ is as in Section
\ref{appl}, and $h:[0,T]\to \R^N$ is a measurable function. Notice
that \eqref{detproblem} may be seen as the special case of
\eqref{GhAGh}, where $\A = C$ and $h=W$.

The following condition is introduced in \cite[Theorem 6.30]{DPZ}
(see also \cite[Proposition 1]{DPIT}) in the time independent case.
Let $(\A(t))_{t\in [0,T]}$ be densely defined and such that $0\in
\varrho(\A(t))$ for all $t\in [0,T]$. Assuming Hypothesis (H2) we
consider the following Hypothesis (K) (which may be weakened
somewhat, cf. \cite[Remark 1.2]{ATito}).

\begin{enumerate}
\item[(K)]\label{H14}
We have $0\in\varrho(\A(t))$ for all $t\in [0,T]$ and there exist uniformly
bounded functions $K_n:[0,T]\to \calL(E)$ such that for all $t\in [0,T]$, all
$n=1, \ldots, N$, and all $x\in \D(B_n)$ we have $B_n \A^{-1}(t) x \in
\D(\A(t))$ and
$$\A(t) B_n \A^{-1}(t) x =  B_n x+ K_n(t)x.$$
\end{enumerate}
The latter may be rewritten as the commutator condition:
$$[\A(t), B_n] \A^{-1}(t) x =  K_n(t)x.$$
In many cases it is enough to consider only $x\in \D(\A(t))$ instead of $x\in
\D(B_n)$ (cf. \cite[Proposition A.1]{ATito}).

Assume that (AT1) and (AT2) hold for the operators $\A(t)$. If {\rm(K)} holds
for the operators $\A(t)$, then the uniform boundedness of $t\mapsto
R(\lambda, \A(t))$ can be used to check that for all $\lambda>0$, {\rm (K)}
holds for the operators $\A(t)-\lambda$ for all $\lambda>0$.

The following lemma lists some consequences of Hypothesis (K).

\begin{lemma}\label{lem:Cons}
Let $(\A(t))_{t\in [0,T]}$ be closed densely defined operators such
that $0\in \varrho(A(t))$ for all $t\in [0,T]$. Assume Hypotheses
{\rm (H2)} and {\rm (K)}.
\begin{enumerate}
\item\label{Consb} For all $n=1, \ldots, N$, $s\in \R$ and $t\in [0,T]$,
$G_n(s)$ leaves $\D(\A(t))$ invariant and
\[\A(t) G_n(s) \A^{-1}(t) = e^{s(B_n + K_n(t))}.\]
\item\label{Consc} For all $R\geq 0$ there is a constant $M_R\geq 0$ such that
for all $n=1, \ldots, N$, $|s|\leq R$ and $t\in [0,T]$ we have
\[\|\A(t) G_n(s) \A^{-1}(t) - G_n(s)\|\leq M_R|s| .\]
\end{enumerate}
\end{lemma}
\begin{proof}
The first assertion follows from the proof of \cite[Proposition 1]{DPIT} and
the second from a standard perturbation result, cf. \cite[Corollary
III.1.11]{EN}.
\end{proof}

We can now formulate a result that related the problems \eqref{nCP} and
\eqref{GhAGh}.
\begin{proposition}\label{DPITresult}
Let $(\A(t))_{t\in [0,T]}$ be closed densely defined operators such
that $0\in \varrho(\A(t))$ for all $t\in [0,T]$. Assume Hypotheses
{\rm (H2)} and {\rm (K)}. Let $h:[0,T]\to \R^N$ be H\"older
continuous with parameter $\alpha\in (0,1]$ and define the similar
operators $(\A_h(t))_{t\in [0,T]}$ as
\[\A_h(t) = G^{-1}(h(t)) \A(t) G(h(t)) \ \ \text{with $D(\A_h(t)) = \{x\in E: G(h(t)) x \in D(\A(t))\}$.}\]
If the operators $\A(t)$ satisfy {\rm (AT1)} and {\rm (AT2)}
with $[(\alpha_1, \beta_1), \ldots, (\alpha_k, \beta_k)]$, then the operators
$\A_h(t)$ satisfy {\rm (AT1)} and {\rm (AT2)} with $[(\alpha_1,
\beta_1), \ldots, (\alpha_k, \beta_k), (\alpha, 0)]$.
\end{proposition}

\begin{proof}
We denote $G_h(t) = G(h(t))$. For all $t\in [0,T]$ and $\lambda
\in\varrho(\A(t))$ we clearly have $\lambda \in\varrho(\A_h(t))$ and
$R(\lambda, \A_h(t)) = G_h^{-1}(t) R(\lambda, \A(t)) G_h(t).$ It
follows from the assumptions on $h$ that for all $t\in [0,T]$,
\[\|R(\lambda, \A_h(t))\|\leq M^2 \|R(\lambda, \A(t))\|,\]
where $M = \sup_{t\in [0,T]} \|G(h(t))\| \vee \|G(-h(t))\|$. Hence
each $\A_h(t)$ is a sectorial operator with the same sector as
$\A(t)$. Thus the operators $\A_h(t)$ satisfy (AT1).

Next we check (AT2). For all $t,s\in [0,T]$ we have
\[\begin{aligned}
\|\A_h(t) & R(\lambda, \A_h(t)) [ \A_h^{-1}(t) - \A_h^{-1}(s)]\|
\\ & = \|G_h^{-1} (t) \A(t) R(\lambda, \A(t)) [\A^{-1}(t) G_h(t) - G_h(t) G_h^{-1}(s) \A^{-1}(s)
G_h(s)]\|
\\ & \leq M \|\A(t) R(\lambda, \A(t)) [\A^{-1}(t) G_h(t) - \A^{-1}(s) G_h(t) ]\|
\\ & \qquad + M \|\A(t) R(\lambda, \A(t)) [\A^{-1}(s) G_h(t) -  G_h(t) G_h^{-1}(s)
\A^{-1}(s) G_h(s)]\|.
\end{aligned}\]
We estimate the two terms on the right-hand side separately. Since
$(\A(t))_{t\in [0,T]}$ satisfies (AT2), it follows for the first term that
\begin{equation}\label{term1}
\begin{aligned}
\|\A(t) R(\lambda, \A(t)) & [\A^{-1}(t) G_h(t) - \A^{-1}(s) G_h(t)] \|
\\ &\leq M\|\A(t) R(\lambda, \A(t)) [\A^{-1}(t) - \A^{-1}(s)  ]\|
\\ & \leq M L \sum_{i=1}^k |t-s|^{\alpha_i} |\lambda|^{\beta_i-1}.
\end{aligned}
\end{equation}
For the second term we have
\begin{equation}\label{rewriteterm2}
\begin{aligned}
\ & \|\A(t) R(\lambda,
\A(t)) [\A^{-1}(s) G_h(t) - G_h(t) G_h^{-1}(s) \A^{-1}(s) G_h(s)]\|
\\ &\qquad \leq
M\|\A(t) R(\lambda, \A(t)) \A^{-1}(s)  [G_h(t) G_h^{-1}(s) - \A(s) G_h(t)
G_h^{-1}(s) \A^{-1}(s)]\|
\\ & \qquad = M\|\A(t) R(\lambda, \A(t)) \A^{-1}(s)  [G(h(t)-h(s)) - \A(s) G_h(t)
G_h^{-1}(s) \A^{-1}(s)]\|.
\end{aligned}
\end{equation}
By an induction argument and Lemma \ref{lem:Cons} as in the proof of
\cite[Theorem 6.30]{DPZ}, the H\"older continuity of $h$ implies that for all
$t,s\in [0,T]$,
\begin{equation}\label{tanabeest}
\|G(h(t) -h(s)) - \A(s) G_h(t) G_h^{-1}(s)\A^{-1}(s)\| \leq M_{\alpha} N |t -
s|^{\alpha}.
\end{equation}
On the other hand it follows from (AT1) and (AT2) that
\begin{equation}\label{AT1AT2comb}
\begin{aligned}
\ & \|\A(t) R(\lambda, \A(t)) \A^{-1}(s)\|
\\ & \qquad \leq \|\A(t)
R(\lambda, \A(t)) [\A^{-1}(s) - \A^{-1}(t)]\| +  \|\A(t) R(\lambda, \A(t))
\A^{-1}(t)\|
\\ & \qquad \leq L \sum_{i=1}^k |t-s|^{\alpha_i} |\lambda|^{\beta_i-1} +
K |\lambda|^{-1}.
\end{aligned}\end{equation}
Combining \eqref{rewriteterm2}, \eqref{tanabeest} and \eqref{AT1AT2comb} gives
\begin{equation}\label{term2}\begin{aligned}
\|\A(t) R(\lambda, \A(t)) & [\A^{-1}(s) G_h(t) - G_h(t) G_h^{-1}(s) \A^{-1}(s)
G_h(s)\| \\ & \leq M L M_{\alpha} N \sum_{i=1}^k |t-s|^{\alpha_i+\alpha}
|\lambda|^{\beta_i-1} + M K M_{\alpha} N |t-s|^{\alpha}
|\lambda|^{-1}\end{aligned}
\end{equation}
We conclude from \eqref{term1}, \eqref{term2}, and the trivial
estimate $|t-s|^{\alpha_i+\alpha} \le C_T|t-s|^{\alpha_i} $ that
\[\big\|\A_h(t)  R(\lambda, \A_h(t)) \big[ \A_h^{-1}(s) - \A_h^{-1}(t)\big]\big\|
\leq \tilde{L} \sum_{i=1}^{k+1} |t-s|^{\alpha_i} |\lambda|^{\beta_i-1}\] for a
certain constant $\tilde{L}$ and $\alpha_{k+1} = \alpha$, $\beta_{k+1} = 0$.
\end{proof}

The main abstract result of this paper reads as follows.

\begin{theorem}\label{mainthm}
Let $E$ be a UMD Banach space and assume that Hypotheses {\rm (H1)},
{\em (H2)}, {\rm (H3)}, and {\rm (H4)} are fulfilled and that {\rm
(AT1)}, {\rm (AT2)}, and {\rm (K)} are satisfied for $C(t) - \mu$ for
some $\mu\in \R$.

\begin{enumerate}
\item The problem \eqref{LinSCP} admits a unique strong solution $U$ on $(0,T]$
for which $AU\in C((0,T];E)$ almost surely.
\item If $u_0\in \D_{A(0)}(1-\sigma,\infty)$ almost surely, then the problem
\eqref{LinSCP} admits a unique strong solution $U$ on $[0,T]$ for which $AU\in
C((0,T];E)$ almost surely. Moreover $AU\in L^p(0,T;E)$ for all $1\leq
p<\sigma^{-1}$.
\item If $u_0\in \D(A(0))$ almost surely, the problem \eqref{LinSCP} admits a
unique strong solution $U$ on $[0,T]$ for which $AU\in C([0,T];E)$ almost
surely.
\end{enumerate}
\end{theorem}
\begin{proof}
If $U_{\mu}$ is a solution of \eqref{LinSCP} with $A(t)$ replaced by $A(t) -
\mu$, then it is easy to see that $t\mapsto e^{\mu t} U_{\mu}(t)$ is a
solution of \eqref{LinSCP}. It follows from this that without loss of
generality we may assume that $\mu = 0$ in the assumptions above.

(1): By the standing assumption made in Section \ref{appl}, the initial value
$u_0$ is an $\F_0$-measurable random variable. By Proposition \ref{DPITresult}
and the H\"older continuity of Brownian motion, the operators $C_W(t)$ satisfy
{\rm (AT1)} and {\rm (AT2)}. Hence by Theorem \ref{thm:AT}, almost surely the
problem \eqref{detproblem} admits a unique classical solution $V$.
To see that $V$ is adapted we argue as follows.

Let $(P_W(t,s))_{0\le s\le t\le T}$ be the
evolution system generated by
$(C_W(t))_{0\le t\le T}$, which exists by virtue of (AT1), (AT2), and the results of
\cite{Acq, AT2}. Then $V(t) =  P_W(t,0) u_0$.
Thus we need to check that for each $t\in [0,T]$ the random variable
$P_W(t,0) u_0$ is strongly $\F_t$-measurable.
Since $u_0:\O\to E$ is strongly $\F_0$-measurable we can approximate
$u_0$ almost surely with $\F_0$-measurable simple functions.
In this way the problem reduces to showing that $P_W(t,0) x$
is $\F_t$-measurable for all $x\in E$. One easily checks that the Yosida approximations
$(C_W^{(n)}(s))_{s\in [0,t]}$ of $(C_W(s))_{s\in [0,t]}$
are strongly $\F_t$-measurable in the strong operator topology.
Moreover, in view of (AT1) and (AT2), $C_W^{(n)}$ is almost surely (H\"older)
continuous in the uniform operator topology.
Therefore by the construction of the evolution family
$P_W^{(n)}(u,s)_{0\leq s\leq u\leq t}$
(for instance via the Banach fixed point theorem (cf. \cite{Paz}))
we obtain that $P_W^{(n)}(t,0) x$ is strongly $\F_t$-measurable.
By \cite[Proposition 4.4]{BaCh}, $P_W(t,0) x = \limn P_W^{(n)}(t,0) x$.
This implies
that $P_W(t,0) x$ is strongly $\F_t$-measurable.

Since $V$ has continuous paths almost surely, it
follows that $V$ is strongly measurable. Since continuous functions are
integrable, the solution $V$ is a strong solution on $(0,T]$. Hence by Theorem
\ref{solequiv}, $U=G_W V$ is a strong solution of \eqref{LinSCP} on $(0,T]$.
The pathwise regularity properties of $V$ carry over to $U$, thanks to (H4).
The pathwise uniqueness of $V$ implies the uniqueness of $U$ again via Theorem
\ref{solequiv} and (H4).

(2): If $u_0\in \D_{A(0)}(1-\sigma,\infty)$ almost surely, then it follows
from $AV\in L^p(0,T;E)$ that $V$ is a strong solution of \eqref{detproblem} on
$[0,T]$. Therefore, Theorem \ref{solequiv} implies that $U$ is a strong
solution of \eqref{LinSCP} on $[0,T]$. The pathwise regularity properties of
$V$ carry over to $U$ as before.

(3): If $u_0\in \D(A(0))$ almost surely, then $V$ is a strong solution of
\eqref{detproblem} on $[0,T]$, and from Theorem \ref{solequiv} we see that $U$
is a strong solution of \eqref{LinSCP} on $[0,T]$. The pathwise regularity
properties of $V$ carry over to $U$ as before.
\end{proof}

\begin{remark}
If $(A(t)-\mu_0)_{t\in [0,T]}$ satisfies (AT1) and (AT2) for a certain
$\mu_0\in \R$, then under certain conditions the perturbation result in
\cite[Lemma 4.1]{DLS} may be used to obtain that $(C(t)-\mu)_{t\in [0,T]}$
satisfies (AT1) and (AT2) as well for $\mu$ large enough. In particular, this
is the case if the $(B_n)_{n=1}^N$ are assumed to be bounded.
\end{remark}

\begin{remark}\label{rem:Bnbounded}
Assume $E$ is reflexive (e.g. $E$ is a UMD space). If the $B_n$ are
bounded and commuting and the closed operators $A(t)-\mu_0$ and
$C(t)-\mu_0$ satisfy (AT1), (AT2) for all $\mu_0\in \R$ large
enough, then (H1) - (H4) are fulfilled. It is trivial that (H2) and
(H3) are satisfied. For (H1) one may use Kato's result (cf.
\cite[Section VIII.4]{Yos}) to check the denseness of the domains.
For (H4) notice that for $\lambda>\mu_0$ (AT1) and (AT2) imply that
$t\mapsto R(\lambda, A(t))$ and $t\mapsto R(\lambda, C(t))$ are
continuous (cf. \cite[Lemma 6.7]{Ta2}). Since $B_n$ are assumed to
be bounded this clearly implies (H4).
\end{remark}

\begin{remark}\label{timereg}
Assume the operators $B_1, B_2, \ldots, B_N$ are bounded and commuting. Then
each $e^{t B_n}$ is continuously differentiable, so $G(W)$ is H\"older
continuous with exponent $\mu\in (0,\frac12)$. As a consequence, time
regularity of the solution $V$ of \eqref{detproblem} translates in time
regularity of the solution $U=G(W) V$ of \eqref{LinSCP}. We will illustrate
this in two ways below.

As in \cite[p. 5]{Schn} it can be seen that if almost surely $u_0\in
D((w-A(0)^{\alpha})$ for some $\alpha\in (0,1]$, then almost surely
$V$ is H\"older continuous with parameter $\alpha$. We conclude that
under the condition that almost surely, $u_0\in D((w-A(0))^{\alpha}$
for some $\alpha\in (0,\frac12)$, $U$ is H\"older continuous with
parameter $\alpha$.

Assume $u_0\in D(A(0))$ and $A(0) u_0\in D_A(\alpha, \infty)$ almost
surely for some $\alpha\in (0, \delta]$. Then we deduce from
\cite[Section 6]{AT2} that almost surely, $C_W V$ has paths in
$C^{\alpha}([0,T];E)$. If $\alpha<\frac12$, then we readily obtain,
almost surely, $A U$ has paths in $C^{\alpha}([0,T];E)$.
\end{remark}

We conclude this section with an example. An non-stochastic version
of the example has been studied in \cite{Acq, Schn,Ya}.

\begin{example}
We consider the problem

\begin{equation}\label{eq:easy2}
\begin{aligned}
D_t u(t,x) & = A(t,x, D) U(t,x) + \sum_{n=1}^N B_n(x)U(t,x)D_t W_n(t), \ \
t\in [0,T], x\in S
\\ V(t, x, D) U &= 0, \ \ t\in [0,T], x\in \partial S,
\\ U(0,x) & = u_0(x), \ \ x\in S
\end{aligned}
\end{equation}
Here
\[A(t,x, D) = \sum_{i,j=1}^d a_{ij}(t,x) D_i D_j  + \sum_{i=1}^d q_{i}(t,x) D_i  + r(t,x), \ \ B_n(x) =  b_n(x),\] and
\[V(t,x) = \sum_{i=1}^d v_i(t,x,D) D_i + v_0(t,x).\]
The set $S\subseteq \R^d$ is a bounded domain with boundary of class $C^2$
being locally on one side of $S$ and outer unit normal vector $n(x)$. We
assume that $\partial S$ consists of two closed (possibly empty) disjoint
subsets $\Gamma_0$ and $\Gamma_1$. Moreover the coefficients are real and
$a_{ij}, q_i, r\in C^{\alpha}([0,T], C(\overline{S}))$, where $\alpha\in
(\tfrac{1}{2},1)$ if $\Gamma_1\neq \emptyset$ and $\alpha\in (0,1)$ if
$\Gamma_1 =\emptyset$ and the matrix $(a_{ij}(\cdot,x))_{i,j}$ is symmetric
and strictly positive definite uniformly in time, i.e. there exists an $\nu>0$
such that for all $t\in [0,T]$ we have
\[\sum_{i,j=1}^d a_{ij}(t,x) \lambda_i \lambda_j \geq \nu \sum_{i=1}^d
\lambda_i^2, \ \ x\in S, \lambda\in \R^d.\] The boundary coefficients are
assumed to be real and $v_i, v_0\in C^{\alpha}([0,T], C^{1}(\partial S))$,
$v_0=1$ and $v_i=0$ on $\Gamma_0$ and there is a constant $\delta>0$ such that
for all $x\in \Gamma_1$ and $t\in[0,T]$ we have $\sum_{i=1}^d v_i(t, x) n_i(x)
\geq \delta$. Finally we assume that $b_n\in C^2(\overline{S})$ and
\begin{equation}\label{condb}
\sum_{i=1}^d v_i(t,x) D_i b_n(x) = 0, \ \  t\in [0,T], x\in \partial S.
\end{equation}
Under these assumptions, for all $p\in (1, \infty)$ and $u_0\in
L^0(\O;\F_0;L^p(S))$ there exists a unique strong solution $U$ of
\eqref{eq:easy2} on $(0,T]$ for which $A U\in C((0,T];L^p(S))$ almost surely.

If $u_0\in
B_{p, \infty,\{V\}}^{2(1-\sigma)}(S)$ almost surely for some $\sigma\in (0,1)$
(see \cite[Section 4.3.3]{Trie} for the definition of this space)
then there exists a unique strong
solution $U$ of \eqref{eq:easy2} on $[0,T]$ for which
almost surely $A U\in C((0,T];L^p(S))$
and $AU\in L^q(0,T;L^p(S))$ for all $1\leq q<\sigma^{-1}$.

Furthermore, if almost surely we have $u_0\in W^{2,p}(S)$  and
$V(0,x) u_0 = 0 \ \ x\in \partial S,$
then there exists a unique strong solution $U$ of \eqref{eq:easy2} on $[0,T]$
for which $A U\in C([0,T];L^p(S))$ almost surely.

Finally, we notice that Remark \ref{timereg} can be used to obtain time
regularity of $U$ and $AU$ under conditions on $u_0$.
\end{example}

\begin{proof}
We check the conditions in Theorem \ref{mainthm}, with $\D(A(t)) =
\{f\in W^{2,p}(S): \ V(t,\cdot)f=0 \hbox{ on }\partial S\}. $ If
$\sigma\not=\frac12(1-\frac1p)$ (which can be assumed without loss
of generality by replacing $\sigma$ by a slighly larger value)
$D_{A(0)}(1-\sigma, \infty) = B_{p, \infty,\{V\}}^{2(1-\sigma)}(S)$,
cf. \cite[Theorem 4.3.3]{Trie}.

It is shown in \cite{Schn}
that for $\lambda_0\in \R$ large enough, (AT1) and (AT2) hold for
$A(t)-\lambda_0$ and $C(t)-\lambda_0$, with coefficients $\alpha$ and
$\beta=\tfrac{1}{2}$ in case $\Gamma_1 \neq \emptyset$ and $\beta = 0$ in case
$\Gamma_1 = \emptyset$. Since the operators $B_n$ are bounded, Remark
\ref{rem:Bnbounded} applies and we conclude that (H1)--(H4) hold.

Let $\lambda>\lambda_0$ be fixed. The only thing that is left to be
checked is condition (K) for the operators $C(t)-\lambda$.
It follows from \eqref{condb} that for all $x\in E$,
$B_n R(\lambda, C(t)) x \in \D(C(t))$. For $n=1, 2, \ldots, N$ and
$t\in[0,T]$ define
\[K_n(t) = (C(t)-\lambda) B_n (C(t)-\lambda)^{-1} - B_n.\]
One can check that $K_n(t) = [C(t), B_n]R(\lambda, C(t))$,
where $[C(t), B_n]$ is the commutator of $C(t)$ and $B_n$.
Since $[C(t), B_n]$
is a first order operator, each $K_n(t)$ is a bounded operator.
To prove their uniform boundedness in $t$, we note that
from the assumptions on the coefficients it follows that there are
constants $C_1, C_2>0$ such that for all $t\in [0,T]$ and $j=1, \ldots, d$,
\[\|R(\lambda, C(t))\|\leq C_1 \ \text{and} \  \|D_j R(\lambda, C(t))\|\leq C_2.\]
Indeed, the first estimate is obviously true, and the second one follows from
the Agmon-Douglis-Nirenberg estimates (see \cite{ADN}).
\end{proof}

\section{The deterministic problem: Tanabe conditions}\label{sec:T}

In the theory for operators $\A(t)$ with time-independent domains $\D(\A(t))
=: \D(\A(0))$ (cf. \cite[Section 5.2]{Ta1}, see also \cite{Am,Lun,Paz}),
condition (AT2) is often replaced by the following stronger condition, usually
called the Tanabe condition,
\begin{enumerate}
\item[(T2)] There are constants $L\geq 0$ and $\mu\in (0,1]$
such that for all $t,s\in [0,T]$ we have
\[\|\A(t)\A^{-1}(0) - \A(s)\A^{-1}(0)\|\leq L |t-s|^{\mu}.\]
\end{enumerate}
It is shown in \cite{Ta1} that condition (T2) implies that there is a constant
$\tilde{L}\geq 0$, such that for all $t,s,r\in [0,T]$ we have
\begin{equation}\label{T2extra}
\|\A(t)\A^{-1}(r) - \A(s)\A^{-1}(r)\|\leq \tilde{L} |t-s|^{\mu}.
\end{equation} In particular the family $\{\A(s) \A^{-1}(t):s,t\in [0,T]\}$
is uniformly bounded.

It is clear that under (H1) and (H3), the operators $A(t)$ satisfy
(T2) if and only if the operators  $C(t)$ satisfy (T2).

\begin{lemma}\label{lem:forH4}
Assume {\rm (H1)}, {\rm (H3)} and that $\D(A(t)) = \D(A(0))$. If $(A(t))_{t\in
[0,T]}$ and $(C(t))_{t\in [0,T]}$ satisfy {\rm (AT1)} and {\rm (T2)}, then
{\rm (H4)} holds.
\end{lemma}
\begin{proof}
Since $\D(A(0)) \subseteq \D(B_n^2)$ and $0\in \varrho(A(0))$, there is a
constant $C_n$ such that $\|B_n^2 x\|\leq C_n\|A(0) x\|$ for all $x\in
D(A(0))$. It follows from the uniform boundedness of $\{A(0) A^{-1}(t):t\in
[0,T]\}$ and \eqref{T2extra} that for all $t,s\in [0,T]$ we have
\[\begin{aligned}
\|B_n^2 A^{-1}(t) - B_n^2 A^{-1}(s)\|
& \leq C_n \|A(0) A^{-1}(t) - A(0)A^{-1}(s)\|
\\ & \leq C_n C \|(A(t)A^{-1}(t) - A(t)A^{-1}(s))\|
\\ & \leq C_n C \|(A(s)A^{-1}(s) - A(t)A^{-1}(s))\|
\leq C_n C \tilde{L}|t-s|^{\mu}.
\end{aligned}\]
This shows that $t\mapsto B_n^2 A^{-1}(t)$ is H\"older continuous. In the same
way one can show that $t\mapsto B_n^2 C^{-1}(t)$ is H\"older continuous. We
 conclude that (H4) holds.
\end{proof}

It is easy to see that the statement in Proposition \ref{DPITresult} holds as
well with (AT2) replaced by (T2) (in the assumption and the assertion).
Thus in the case where the domains $\D(A(t))$ are constant,
the more difficult Acquistapace-Terreni theory is not needed.

If the operators $B_1, \ldots, B_N$ are bounded we have the following
consequence of Theorem \ref{mainthm}. Note that the assumptions are made on the operators $A(t)$ rather than on $C(t)$.

\begin{proposition}\label{prop:bdd}
Let $E$ be a UMD space and $\D(A(t)) = \D(A(0))$. Assume that the operators
$A(t)-\lambda$ satisfy {\rm (AT1)} and {\rm (T2)} for all $\lambda\in \R$ large
enough, and let $B_1, \ldots, B_N\in \calL(E)$ be bounded commuting operators which leave $\D(A(0))$ invariant. Consider the problem
\begin{equation}\label{bounded}
\begin{aligned}
d U(t) &= A(t) U(t) \, dt + \sum_{n=1}^N B_n U(t) d W_n(t), \ \ t\in [0,T],
\\ U(0) & = u_0.
\end{aligned}
\end{equation}
\begin{enumerate}
\item If $u_0\in E$ almost surely, the problem \eqref{bounded} admits a unique
strong solution $U\in C([0,T];E)$ on $(0,T]$ for which $A U\in C((0,T];E)$.
\item If $u_0\in \D_{A(0)}(1-\sigma,\infty)$ almost surely, then the problem
\eqref{LinSCP} admits a unique strong solution $U\in C([0,T];E)$ on $[0,T]$
with $AU\in C((0,T];E)$. Moreover $AU\in L^p(0,T;E)$ for all $1\leq
p<\sigma^{-1}$.
\item If $u_0\in \D(A)$ almost surely, the problem \eqref{LinSCP} admits a unique
strong solution $U\in C([0,T];E)$ on $[0,T]$ for which $AU\in C([0,T];E)$.
\end{enumerate}
\end{proposition}

\begin{proof}
We check the conditions of Theorem \ref{mainthm}. It follows from Remark
\ref{rem:Bnbounded} that (H1), (H2) and (H3) are satisfied. Lemma
\ref{lem:forH4} implies that (H4) is satisfied.

By the bounded perturbation theorem, for $\lambda\in \R$ large enough the
operators  $C(t)-\lambda=A(t)-\frac{1}{2} \sum_{n=1}^N B_n^2-\lambda$ satisfy
(AT1). Hence for $\lambda$ large enough, condition (T2) for the operators
$C(t)-\lambda$ follows from (T2) for the operators $A(t)-\lambda$.

Finally to check (K), by the assumption on the
operators $B_n$ we have $\D(A(0))= \D(C(0))$, and
by the closed graph theorem we have $\|B_n x\|_{\D(C(0))}\leq
c_n\|x\|_{\D(C(0))}$ for some constant $c_n$. This implies that $\|C(0) B_n
x\|\leq c_n\|C(0)x\|$. We check that the operators $K_n(t) = C(t) B_n
C^{-1}(t) - B_n$ are uniformly bounded. By the remark
following \eqref{T2extra}, the family $\{C(0) C^{-1}(t): t\in [0,T]\}$ is uniformly bounded, say by some constant $k$, and  therefore
\[
\bal
\|C(t) B_n C^{-1}(t)\| & \leq \|C(t) C^{-1}(0) C(0) B_n C^{-1}(0) C(0) C^{-1}(t)\|\\ & \leq k^2 \|C(0) B_n C^{-1}(0)\|\leq c_n.
\eal\]
\end{proof}

Next we return to the problem \eqref{eq:zakai} discussed at the beginning of the paper.
\begin{example}\label{ex:zakai}
We consider the problem
\begin{equation}\label{Zakai}
\begin{aligned}
D_t u(t,x) & = A(t,x, D) U(t,x) + B(x,D) D_t W(t), \ \ t\in [0,T], x\in \R^d
\\ U(0,x) & = u_0(x), \ \ x\in \R^d
\end{aligned}
\end{equation}
Here
\[
\bal A(t,x, D) &= \sum_{i,j=1}^d a_{ij}(t,x) D_i D_j  + \sum_{i=1}^d q_{i}(t,x) D_i  + r(t,x), \\
 B(x,D) & =  \sum_{i=1}^d b_{i}(x) D_i + c(x).
\eal\] All coefficients are real-valued and we take $a_{ij}, q_i, r$
uniformly bounded in time with values in $C_b^1(R^d))$. The
coefficients $a_{ij}, q_i$ and $r$ are $\mu$-H\"older continuous in
time for some $\mu\in (0,1]$, uniformly in $\R^d$. Furthermore we
assume that the matrices $(a_{ij}(t,x))_{i,j}$ are symmetric, and
there exists a constant $\nu>0$ such that for all $t\in [0,T]$
$$
\sum_{i,j=1}^d \Big(a_{ij}(t,x) -\frac12 b_i(x) b_j(x)\Big) \lambda_i
\lambda_j \geq \nu \sum_{i=1}^d \lambda_i^2, \ \ x\in \R^d, \lambda\in \R^d.
$$
Finally, we assume that $b_i,c\in C^2_b(\R^d)$. Under these
assumptions it follows from Theorem \ref{mainthm} that for all $p\in (1,
\infty)$ and $u_0\in L^0(\O, \F_0;L^p(\R^d))$, there exists a unique
strong solution $U$ of \eqref{Zakai} on $(0,T]$ with paths in
$C([0,T];L^p(\R^d))\cap C((0,T];W^{2,p}(\R^d))$.
If moreover $u_0\in B_{p, \infty}^{2(1-\sigma)}(\R^d)$
almost
surely, then there exists a unique strong solution $U$ of
\eqref{eq:easy2} on $[0,T]$ for which $U\in C((0,T];W^{2,p}(\R^d))$
almost surely and $AU\in L^q(0,T;L^p(\R^d))$ for all $1\leq
q<\sigma^{-1}$. If $u_0\in W^{2,p}(\R^d)$ almost surely, then there
exists a unique strong solution $U$ of \eqref{Zakai} on $[0,T]$ with
paths in $C^{\alpha}([0,T]; L^p(\R^d)) \cap C([0,T];W^{2,p}(\R^d))$
for all $\alpha\in (0,\tfrac{1}{2})$.
\end{example}

In \cite{Brz1}, for $A(t)\equiv A$ a strong solution on $[0,T]$ with paths in
$L^2(0,T;W^{2,p}(\R^d))$ almost surely is obtained for initial data satisfying
$u_0\in  B_{p, 2}^{1}(\R^d)$ almost surely.
In \cite{Kry} it
is assumed that $u_0\in H_p^{2-\frac{2}{p}}(\R^d)$ and a solution is obtained
with paths in $L^p(0,T;W^{2,p}(\R^d))$ almost surely.

\begin{proof}
Let $E= L^p(\R^d)$, where $p\in (1, \infty)$. Let $\D(A(t)) = W^{2,p}(\R^d)$
and $A(t) f)= A(t, \cdot, D) f$ for all $t\in [0,T]$. Let $\D(B_0) =
W^{1,p}(\R^d)$ and $B_0 f= B(\cdot, D) f$, and let $(B, D(B))$ be the closure
of $(B_0, D(B_0))$. Note that by real interpolation we have
$B_{p, \infty}^{2(1-\sigma)}(\R^d) = D_A(1-\sigma, \infty)$, see \cite{Trie}.

We check the conditions of Theorem \ref{mainthm}. We begin with the Hypotheses
(H1)-(H3). That (H1) holds is clear, and  (H2) follows as in \cite[Example
C.III.4.12]{Ar}. Finally (H3) follows from $\D(A(t)) \subseteq \D(B^2)$.

The operators $A(t)-\lambda$ and $C(t)-\lambda$
satisfy condition (AT1) for all $\lambda\in \R$ large enough (cf.
\cite[Section 3.1]{Lun}). Furthermore it can be checked that $A(t)-\lambda$ and
$C(t)-\lambda$ satisfy (T2). Now Condition (H4) follows from
\eqref{lem:forH4}.

To check (K) for the operators $C(t)-\lambda$, put $K(t) = [C(t), B]R(\lambda, C(t))$.  Since the third order derivatives in the commutator $[C(t), B]$ cancel and
$a_{ij}(t), q_i(t), r(t)\in C_b^1(R^d)$ and $b_i,c\in C^2_b(\R^d)$,
the operators $K(t)$ are bounded for each $t\in [0,T]$. Moreover,
\[K(t) = [C(t), B]R(\lambda, C(t)) = [C(t)-\lambda, B]R(\lambda, C(t))=(C(t)-\lambda) B (C(t)-\lambda)^{-1} +B\]
on $W^{1,p}(\R^d)$, and this identity extends to $\D(B)$ (see
\cite[Proposition A.1]{ATito}). To check that $K$ is uniformly bounded, note
that by the uniform boundedness of the family $(\lambda - C(0)) R(\lambda,
C(t))$ it suffices to check that there is a constant $C$ such that for all
$t\in [0,T]$ and $f\in W^{2,p}(\R^d)$,
\[\|[C(t), B] f\| \leq C\|f\|_{W^{2,p}(\R^d)}.\]
But this follows from the assumptions $a_{ij}, q_i, r\in \mathcal{L}^{\infty}([0,T];C_b^1(R^d))$
and $b_i,c\in C^2_b(\R^d)$.

Finally, we show that if $u_0\in W^{2,p}(\R^d)$ almost surely, then
$U$ has paths in $C^{\alpha}([0,T];L^p(\R^d))$ for all $\alpha\in
(0,\tfrac{1}{2})$. One can check that for all $x\in D(A(0))$, $G(t)
x$ is continuously differentiable and there are constants $C_1, C_2$
such that for all $x\in D(A(0))$ and $s,t\in [0,T]$,
\[\|G(t) x - G(s)x\|\leq C_1|t-s| \|x\|_{\D(A(0))}\leq C_2|t-s| \|x\|_{\D(C_W(0))}.\]
On the other hand it follows from Theorem \ref{thm:AT} that \eqref{detproblem}
has a unique strict solution $V$. It follows that there exist maps $M,
M_\alpha:\O\to \R$ such that all for $s, t\in [0,T]$
\[\begin{aligned}
\|U(t) - U(s)\|& \leq \|G_W(t) V(t) - G_W(s) V(s)\| \\ &  \leq \|G_W(t) V(t) -
G_W(t) V(s)\| + \|G_W(t) V(s) - G_W(s) V(s)\| \\ &   \leq M \|V(t)-V(s)\| +
M_{\alpha} |t-s|^\alpha \|V(s)\|_{\D(C_W(0))}. \end{aligned}\] The first term
can be estimated because $V$ is continuously differentiable. We already
observed that $(C_W(s)-\mu)_{s\in [0,T]}$ satisfies (T2) for $\mu$ large. In
particular $(C_W(0) - \mu)(C_W(s)-\mu)^{-1}$ is uniformly bounded in $s\in
[0,T]$. Since $s\mapsto C_W(s) V(s)$ and $V$ are uniformly bounded, we
conclude that $\|V(s)\|_{\D(C_W(0))}$ is uniformly bounded. The result follows
from this.
\end{proof}

\section{Wong-Zakai approximations}

As has been shown in \cite{BCF} for a related class of problems in a
Hilbert space setting, the techniques of this paper can be used to prove
Wong-Zakai type approximation results for the problem \eqref{eq:abstr-zakai},
$$\begin{aligned}
d U(t) &= A(t) U(t) dt + B U(t)\, d W(t), \qquad t\in [0,T],
\\ U(0) & = u_0.
\end{aligned}
$$
and possible generalizations for time-dependent operators $B(t)$.
We shall briefly sketch the main idea and defer the details to a forthcoming
publication.

Let $W^{(n)}$ be adapted processes with $C^1$ trajectories
such that almost surely, $\limn W_n = W$ uniformly on $[0,T]$ and consider
the problem
\begin{equation}\label{eq:abstr-zakai-approx}
\begin{aligned}
d U_n(t) &= (A(t)-\frac12B^2) U_n(t) dt + B U_n(t)\, d W_n(t), \qquad t\in [0,T],
\\ U(0) & = u_0.
\end{aligned}
\end{equation}
This equation may be solved path by path as follows.
Under the assumptions made in Section \ref{appl} and using the notations
introduced there,
define $$C_{W_n}(t,\o) := G^{-1}(W_n(t,\o))C(t)  G(W_n(t,\o))
$$
and consider the pathwise deterministic
problem
\begin{equation}\label{detproblem-approx}
\begin{aligned}
V_n'(t) & = C_{W_n}(t) V_n(t), \ \ t\in[0,T],
\\ V_n(0) &= u_0.
\end{aligned}
\end{equation}
Arguing as in the proof of Theorem \ref{solequiv},  $U_n := G(W_n)V_n$ is a strong
solution of \eqref{eq:abstr-zakai-approx} if and only if $V_n$ is a strong
solution of \eqref{detproblem-approx}, the difference being that instead of the
It\^o formula the ordinary chain rule is applied; this accounts for the loss of
a factor $\frac12B^2$.

In analogy to \cite[Theorems 1 and 2]{BCF},
under suitable conditions on the operators $A(t)$ and $B$ such as given in Sections
\ref{sec:AT} and \ref{sec:T} it can be shows that
$\limn V_n = V$ almost surely, where $V$ is the strong solution of
\eqref{detproblem} and the almost sure convergence takes place in the functional space to
which the trajectories of $V$ belong. It follows that
$\limn U_n = U$ almost surely,  where $U$ is the strong solution of
\eqref{eq:abstr-zakai} and again the almost sure convergence takes place in the functional space to
which the trajectories of $U$ belong.

\medskip

{\em Acknowledgment} -- The authors thank Roland Schnaubelt for useful
discussions which clarified some technical issues connected with
the Acquistapace-Terreni conditions, and
the anonymous referee for the detailed suggestions which led to some
improvements in the presentation.

\def\cprime{$'$}
\providecommand{\bysame}{\leavevmode\hbox
to3em{\hrulefill}\thinspace}

\end{document}